\documentclass[preprint]{imsart}

\usepackage[OT1]{fontenc}
\usepackage{amsthm,amsmath,amssymb,mathrsfs,enumerate}
\usepackage[numbers]{natbib}
\usepackage[colorlinks,citecolor=blue,urlcolor=blue]{hyperref}
\usepackage{dsfont,enumitem}

\numberwithin{equation}{section}
\newtheorem{theorem}{Theorem}
\newtheorem{corollary}[theorem]{Corollary}
\newtheorem{lemma}[theorem]{Lemma}
\newtheorem{proposition}[theorem]{Proposition}
\theoremstyle{remark}
\newtheorem{remark}[theorem]{Remark}
\newtheorem{example}{Example}

\newcommand{\bv}[1]{\boldsymbol{#1}}
\newcommand{\eqdist}{\overset{\mathcal{D}}{=}}
\newcommand{\convd}{\overset{\mathcal{D}}{\to}}
\newcommand{\convp}{\overset{\mathbb{P}}{\to}}
\newcommand{\convas}{\overset{\text{a.s.}}{\to}}
\newcommand{\convw}{\overset{v}{\to}}
\newcommand{\cond}[2]{\left.#1\,\right| #2}
\newcommand{\cov}{\mathrm{Cov}}
\newcommand{\var}{\mathrm{Var}}
\newcommand{\ind}{\mathds{1}}
\newcommand{\diag}{\mathrm{diag}}
\newcommand{\iidsim}{\overset{\text{iid}}{\sim}}
\newcommand{\n}{{(n)}}
\newcommand{\law}{\mathscr{L}}
\newcommand{\Rad}{\mathrm{Rad}}
\newcommand{\mathd}{\mathrm{d}}
\newcommand{\bigO}{\mathcal{O}}

\begin{document}

\begin{frontmatter}
\title{Normal approximations for discrete-time occupancy processes}
\runtitle{Approximations for occupancy processes}

\begin{aug}
\author{\fnms{Liam} \snm{Hodgkinson}\thanksref{t1}\thanksref{t2}\ead[label=e1]{liam.hodgkinson@uqconnect.edu.au}},
\author{\fnms{Ross} \snm{McVinish}\thanksref{t2}\ead[label=e2]{r.mcvinish@uq.edu.au}},
\and
\author{\fnms{Philip K.} \snm{Pollett}\thanksref{t2}\ead[label=e3]{pkp@maths.uq.edu.au}}

\thankstext{t1}{Supported by an Australian Postgraduate Award.}
\thankstext{t2}{All authors are supported in part by the Australian Research Council (Discovery Grant DP150101459 and the ARC Centre of Excellence for 
Mathematical and Statistical Frontiers, CE140100049).}

\runauthor{L. Hodgkinson, R. McVinish, P.K. Pollett}

\affiliation{University of Queensland}

\address{
School of Mathematics and Physics\\
University of Queensland\\
St. Lucia, Brisbane \\
Queensland, 4072\\
Australia \\
\printead{e1}\\
\printead*{e2}\\
\printead*{e3}
}

\end{aug}

\begin{abstract}
We study normal approximations for a class of discrete-time occupancy processes, namely, Markov chains with transition kernels of product Bernoulli form. This class encompasses numerous models which appear in the complex networks literature, including stochastic patch occupancy models in ecology, network models in epidemiology, and a variety of dynamic random graph models. Bounds on the rate of convergence for a central limit theorem are obtained using Stein's method and moment inequalities on the deviation from an analogous deterministic model. As a consequence, our work also implies a uniform law of large numbers for a subclass of these processes.
\end{abstract}



\end{frontmatter}

\maketitle

\section{Introduction}

Treating a complex system as a large collection of interacting entities has become a standard modelling paradigm \citep{Bonabeau:2002aa, Grimm:2005aa,Liggett:1997, Liggett:2012aa}. Growing interest in binary interacting particle systems and agent-based modelling, where entities are treated as nodes with a binary state, has led to the development of general and highly detailed discrete-time models of use in a wide variety of fields.  Ecologists have been captivated by the capacity of \emph{stochastic patch occupancy models} (SPOMs) to help explain the influence of spatial heterogeneity on population dynamics \citep{Hanski:1994,Hanski:2003aa,MacKenzie:2002aa}.  \emph{Probabilistic cellular automata} (PCA) have enjoyed similar popularity in statistical mechanics \citep{Dobrushin:1990aa,TO:15,Wolfram:1983aa}, and \emph{network models} and related \emph{random graph models} have seen numerous applications in epidemiology, and social
and computer science  \citep{Boccaletti:2014aa,Bollobas:2001,ChungLu:2006,Durrett:2007}.

Encompassing many of these models is a class of processes, called here \emph{occupancy processes}, whose state records the occupancy at each of $n$ nodes, and whose transitions at the various nodes are independent conditional on the state. More precisely, an occupancy process is a discrete-time Markov chain $\bv{X}_t=(X_{1,t},\dots,X_{n,t})$, $t=0,1,\dots$, taking values in $\{0,1\}^n$  ($1$ denoting occupancy) such that, given $\bv{X}_t = \bv{x}$, 
$X_{1,t+1},\dots,X_{n,t+1}$ are independent. Under this assumption,
the transition probabilities of $\bv{X}_t$ are given in terms
of functions $P_{i,t}:\{0,1\}^n \to [0,1]$ ($i=1,\dots,n;\ t=0,1,\dots$)
given by
\[
\mathbb{P}\left(\cond{X_{i,t+1}=1}{\bv{X}_t = \bv{x}}\right) = P_{i,t}(\bv{x}).
\]
This includes all finite PCA (occupancy processes need not be local)
and all SPOMs, as well as any network model
where individuals behave independently in the short-term. 
In the random graph framework, 
the nodes become the edges of the graph, with occupancies dictating the
corresponding adjacency matrix. Here, the occupancy process subset of dynamic random graphs are those which evolve at each time-step according to a set of
edge-independent rules, which can depend (in an arbitrary way)
on the state of the graph at the previous time-step.

Borrowing from the terminology of cellular automata \citep{Wolfram:1983aa}, we refer to the collection of functions $P_t = (P_{i,t})_{i=1}^n$, $t=0,1,\dots$, as the \emph{global rule} of $\bv{X}_t$, and to each $P_{i,t}$ as the \emph{local rule} of $X_{i,t}$.  It is convenient to write each local rule in terms of functions $S_{i,t}, C_{i,t}: \, \{0,1\}^n \to [0,1]$,
called the \emph{survival} and \emph{colonization
functions} respectively, satisfying 
\begin{equation}
P_{i,t}(\bv{x}) = x_i S_{i,t}(\bv{x}) + (1 - x_i) C_{i,t}(\bv{x}),
\qquad
\bv{x}\in\{0,1\}^n.
\label{eq:TransDecomp}
\end{equation}
The increased precision afforded by an occupancy process often comes at the expense of tractability. Models of practical interest are not usually amenable to traditional finite-state Markov chain analysis, for the state space is often prohibitively large. Even the efficient simulation of very large systems of this kind presents an ongoing challenge \citep{Brand:2015aa}. Instead, it is common to rely on approximations, the global rule suggesting a natural deterministic model for the evolution of occupancy probabilities (see for example \citep{Ovaskainen:2001aa}). Assuming the domain of $P_t$ is extended to the interior of the hypercube $[0,1]^n$, we may define $\bv{p}_t = (p_{1,t},\dots,p_{n,t})$ by
\begin{equation}
p_{i,t+1} = P_{i,t}(\bv{p}_t),
\qquad i=1,\dots,n. 
\label{eq:Detalternative}
\end{equation}
For the extension of the global rule, assume (I) $P_t\in\mathcal{C}^{3}$,  providing some regularity to the approximation, and 
(II), so that~(\ref{eq:TransDecomp}) holds with $S_{i,t}$, $C_{i,t}$ independent
of their $i$-th argument, $\partial_{i}^{2}P_{i,t}\left(x\right)=0$ for each $i=1,\dots,n$, where $\partial_i$ is the partial derivative with respect to the $i$-th component. We assume throughout that $\bv{X}_0$ is fixed
and $\bv{p}_0 = \bv{X}_0$, although the case where the $X_{i,0}$ are independent random variables with  
$\mathbb{P}(X_{i,0}=1) = p_{i,0}$ can be treated by taking $P_{i,0}(\bv{x}) = p_{i,0}$ and starting the process from time $t = 1$ instead. Additional assumptions would be required to treat dependent $\bv{X}_0$.

One clear advantage of working with~(\ref{eq:Detalternative}) is that the long-term dynamics are easier to elucidate, especially in the time-homogeneous case (this is discussed in greater detail in Section \ref{sec:LongTerm}). 
On the other hand, (\ref{eq:Detalternative}) captures none of the
variability present in the original system, limiting its applicability as a predictive
model. To address this, we instead consider a distributional approximation,
analysing the fluctuations of $\bv{X}_t$ about its deterministic approximation.
In light of the conditional
independence feature of the occupancy process,
it would seem reasonable to expect that these fluctuations are approximately normal. Define the autoregressive Gaussian process $\bv{Z}_t=(Z_{1,t},\dots,Z_{n,t})$ by the recursion
\begin{equation}
Z_{i,t} = p_{i,t} + \sum_{j=1}^n \partial_j P_{i,t}(\bv{p}_{t-1}) (Z_{j,t-1}-
p_{j,t-1}) + z_{i,t} \sqrt{p_{i,t}(1-p_{i,t})},
\end{equation}
where each $z_{i,t}$ is an independent standard normal random variable. 
Letting $\zeta_t = n^{-1/2}(\bv{X}_t - \bv{p}_t)$ and $\xi_t = n^{-1/2}(\bv{Z}_t - \bv{p}_t)$ denote the normalised fluctuations of $\bv{X_t}$
and $\bv{Z}_t$ about $\bv{p}_t$, we show under a few additional assumptions that, for large $n$, the projections $\langle \zeta_t, h \rangle$ and $\langle \xi_t, h \rangle$ over $h \in \mathbb{R}^n$, are close in law. For a distributional approximation of $f(\bv{X}_t)$
where $f \in \mathcal{C}^3([0,1]^n)$ is arbitrary, we appeal to
the linear approximation 
\begin{equation}
\label{eq:LinApproxGen}
f(\bv{X}_t) \approx f(\bv{p}_t) + \sqrt{n}\langle \zeta_t, \nabla f(\bv{p}_t)\rangle.
\end{equation}
An $L^1$ estimate for the error in (\ref{eq:LinApproxGen}) 
in terms of the derivatives of $f$ is provided in Proposition~\ref{prop:LinEstimate}.

The cross-covariances of $\langle \xi_t, h \rangle$ are established by polarization: for each
$0 \leq s \leq t$, letting $D_{s,t} = D_s \cdots D_{t-1}$ with $D_t = DP_t(\bv{p}_t)^\top$ (interpreting the empty product as unity),
\[
\cov[\langle \xi_s,h\rangle,\langle \xi_t, h'\rangle] = \sum_{r=1}^{s\wedge t} \sigma_r [D_{r,s} h,\,
D_{r,t} h'],
\]
where $\sigma_t$ is the symmetric bilinear form defined for $h,h' \in \mathbb{R}^n$
by
\[
\sigma_t[h,h'] := \frac1n \sum_{i=1}^n h_i h_i' \, p_{i,t} (1 - p_{i,t}).
\]
The quadratic case $\sigma_t[h,h] =:\,\sigma_t^2[h]$ provides the approximate
variance introduced in the $t$-th step of the occupancy process; indeed
$\var\langle \xi_t, h \rangle = \sigma_t^2[h] + \var\langle\xi_{t-1}, D_{t-1} h\rangle$ for each $t \geq 1$.

Our analysis proceeds via entirely non-asymptotic methods.
For any $1 \leq q \leq \infty$ and arbitrary $h \in \mathbb{R}^n$, 
we exploit the celebrated method of Stein~\citep{Chen:2010aa} to 
bound the difference in law between
$\langle \zeta_t, h \rangle$ and $\langle \xi_t, h \rangle$ under the $L^q$ metric (defined
for
random variables $X$ and $Y$ by $\|\law(X)-\law(Y)\|_q$, where
$\law(X)$ denotes the distribution function (or law) of $X$). As a consequence
of H\"{o}lder's inequality and integration by parts, if $q^{-1}+r^{-1} = 1$,
then
\begin{equation}
\|\law(X)-\law(Y)\|_q = \sup\{|\mathbb{E}g(X)-\mathbb{E}g(Y)|\,:\,\|g'\|_r
\leq 1\},
\end{equation}
where the derivative $g'$ is understood in the absolutely continuous sense.
For particular choices of $q$ this reduces to other commonly used metrics;
for example, the case $q = 1$ coincides with the \emph{Wasserstein
metric}:
by the Kantorovich-Rubinstein formula \citep{Kantorovich:1958aa}
\[
\|\law(X)-\law(Y)\|_1 = \inf_{X' \sim X, Y' \sim Y} \mathbb{E}|X'-Y'|,
\]
where the infimum is taken over all couplings $(X',Y')$ of $X$ and $Y$. In
contrast, the case $q = \infty$,
\[
\|\law(X)-\law(Y)\|_{\infty} = \sup_{x \in \mathbb{R}}
|\mathbb{P}(X \leq x) - \mathbb{P}(Y \leq x)|
\]
is the \emph{Kolmogorov metric}. 

Denoting by $\|\cdot\|_{\infty}$ the supremum norm, define,
for each $t=0,1,\dots$, quantities
\begin{align*}
\alpha_t & = \max_{j=1,\dots,n} 
\sum_{\substack{i=1 \\ i\neq j}}^{n} \| \partial_{j} P_{i,t} \|_{\infty},
 & \beta_t^2 & = \frac1n \sum_{\substack{i,j=1 \\ i\neq j}}^n \|\partial_j P_{i,t}\|_{\infty}^2, \\
\Gamma_t & = \max_{j,k=1,\dots,n} \sum_{i=1}^n \|\partial_j \partial_k P_{i,t}\|_{\infty}, & 
\gamma_t & = \frac1n\sum_{i,j=1}^n \|\partial_j^2 P_{i,t}\|_{\infty}, \\
\delta_t &= \max_{j=1,\dots,n}\sum_{i,k=1}^n\|\partial_j \partial_k^2 P_{i,t}\|_{\infty}, & \alpha_{s,t} & = \sum_{r=s+1}^{t-1} \alpha_{r},
\end{align*}
and let $\psi_t = \beta_t + \gamma_t$. 
The quantities $ \alpha_{t} $ and $ \beta_{t} $ are not altogether unusual, as the speed of the system and the dependence between nodes is well quantified in the derivatives of each local rule. For systems whose general dynamics strongly depend upon the state of a single node, $\alpha_t$ will be large, and so there is little hope in expecting the deterministic process~(\ref{eq:Detalternative}) to be representative. This is the case for the mainland-island metapopulation model previously studied by the authors \citep{McVinish:2012aa}, which is instead well-approximated by a \emph{semi-deterministic} system. Moreover, $\gamma_t$, $\Gamma_t$, and $\delta_t$ provide some essential measures of the regularity of the global rule.

Our main approximation result is encapsulated in Theorem
\ref{thm:WasserBound}.

\begin{theorem}
\label{thm:WasserBound}
There exists a universal constant $C > 0$ such that, 
for any $h \in \mathbb{R}^n$,
any integer $t \geq 1$ and $1 \leq q \leq \infty$,
\begin{equation}
\label{eq:GenWasserBound}
\|\law\langle \zeta_t, h \rangle-\law\langle \xi_t, h \rangle\|_q \leq
C \|h\|_{\infty}^{4-1/q} \sqrt{\frac{1+\log n}{n}} \sum_{s=0}^{t-1} 
\frac{\kappa_s \cdot e^{(4-1/q)\alpha_{s,t}}}{\sigma_{s+1}^{4-2/q}[D_{s+1,t}h]},
\end{equation}
where, for every $t=0,1,\dots$,
\[
\kappa_{t}=(1+\alpha_t+n\Gamma_{t}+n^{1/2}\delta_t)\sum_{s=0}^{t-1}[1+\psi_s\sqrt{n}
(1+\psi_s\sqrt{n})]t e^{16\alpha_{s,t}}.
\]
\end{theorem}

The logarithmic factor in (\ref{eq:GenWasserBound}) seems to arise only
at complete generality. For many special cases, including certain
mean-field models, the order may be improved to the optimal 
rate of convergence $\bigO(n^{-1/2})$ with minimal effort.
An obvious example is the case where each node transitions without interactions 
(so that $\psi_t = 0$), as a consequence of the Berry-Esseen bound
\citep[Theorem 3.6, Corollary 4.2]{Chen:2010aa}. 
We explore this idea further in Remark~\ref{rem:MeanField}. 

As an important corollary, we
provide a general central limit result for a subclass of approximable
occupancy processes. Consider a sequence of occupancy
processes $\{\bv{X}_t^\n\}_{n=1}^{\infty}$, $t=0,1,\dots,$ with
corresponding global rules $\{P_t^\n\}_{n=1}^{\infty}$ with sequences
$\{\alpha_t^\n,\psi_t^\n,\Gamma_t^\n,\delta_t^\n\}_{n=1}^{\infty}$ for
each $t \geq 0$,
and one-step variances
$\{\sigma_{n,t}^2[h]\}_{n=1}^{\infty}$. 
Extending $\langle \zeta_t, h \rangle$ to
$h \in \ell^{\infty}$ in the obvious way, we obtain

\begin{corollary}
\label{cor:CentralLim}
Let $t \geq 1$, and suppose that
$\sup_n \alpha_s^\n < \infty$, $\sup_n \delta_s^\n < \infty$,
$\beta_s^\n = \bigO(n^{-1/2})$, and $\Gamma_s^\n = \bigO(n^{-1})$ as
$n \to \infty$, for each $s \leq t$. Suppose also that,
for each $s\leq t$, there 
are continuous functions $\sigma_s^2:\,\ell^{\infty}\to[0,\infty)$ and
$\mathcal{J}_s:\,\ell^{\infty}\to\ell^{\infty}$ such that
$\sigma_{n,s}^2[h] \to \sigma_s^2[h]$ as $n\to\infty$ and
\begin{equation}
\lim_{n\to\infty} \sum_{i=1}^n h_i \partial_j P_{i,s}^\n(\bv{p}_s) = 
(\mathcal{J}_s h)_j,
\end{equation}
for every $j=1,2,\dots$ Then, for any $h_1,\dots,h_m \in \ell^{\infty}$ and
$t_1,\dots,t_m$,
\[
(\langle\zeta_{t_1}^\n,h_1\rangle,\dots,\langle\zeta_{t_m}^\n, h_m\rangle) 
\convd \mathcal{N}(0,(\Sigma_{t_i,t_j}[h_i,h_j])_{i,j=1}^m),
\]
where
\[
\Sigma_{s,u}[h_s,h_u] = \begin{cases}
\mathcal{V}_s[h_s,\mathcal{J}_s\circ\cdots
\mathcal{J}_{u-1}h_u], & \quad s \leq u \\
\Sigma_{u,s}[h_u,h_s], & \quad s > u, \\
\end{cases}
\]
and $\mathcal{V}_t:\,\ell^{\infty}\times\ell^{\infty}\to[0,\infty)$ 
is defined by
$\mathcal{V}_0[h_1,h_2] = 0$ and
\[
\mathcal{V}_{t+1}[h_1,h_2] = \sigma_{t+1}[h_1,h_2] + 
\mathcal{V}_t[\mathcal{J}_t h_1,\mathcal{J}_t h_2], \quad t \geq 0,
\] 
with $\sigma_t[h_1,h_2] = \frac12 (\sigma_t^2[h_1+h_2]-\sigma_t^2[h_1-h_2])$.
\end{corollary}

Classical examples of models satisfying the assumptions of Corollary~\ref{cor:CentralLim}
include those which involve
any form of mean-field assumption.
In fact, for homogeneous systems, if $\alpha_t$ is
bounded away from zero, then to satisfy the assumptions
of Corollary~\ref{cor:CentralLim}, each local rule must
depend on a non-zero proportion of the whole system.
Of course, there are many notable types of occupancy processes in the
literature that do not satisfy these assumptions.
In particular, this effectively rules out the
majority of integrable probabilistic cellular automata (which are characterised by strict locality~\citep{Dobrushin:1990aa}) from our analysis. However, we do not expect even the deterministic process to be representative in these cases.
As an example, consider the Domany-Kinzel probabilistic cellular automata (PCA)
on the discrete torus of length $n$:
\[
P_i(\bv{x}) = (q_2 - q_1)x_i x_{i+1} + q_1 (1 - x_i) x_{i+1},\quad q_1,q_2 \in [0,1],
\]
where $x_{n+1} := x_1$. The application of the deterministic
approximation (\ref{eq:Detalternative}) (called the \emph{one-site mean-field
approximation} in the relevant literature) to this system suggests the
existence of a phase transition at $q_1 = \frac12$. On the other hand, a
two-site approximation (see \citep[\S15.2]{TO:15}) suggests the location of the phase transition along $q_1$ varies
according to $q_2$ -- at $q_2 = 0$, the transition occurs at $q_1 = \frac23$.
Indeed, the behaviour of the Domany-Kinzel PCA has been determined quite
well via numerics; at $q_2 = 0$, the phase transition has been estimated
at $q_1 \approx 0.8$ to at least three decimal precision \citep{Zebende:1994aa}. The approximation
does not fare better in the short term either. If the $X_{i,0}$ are independent
and identically distributed with mean $p_0$, $\bv{X}_t$ is an exchangeable
random vector for each $t=1,2,\dots.$ In this case, one can verify by direct
calculation that, for any $n$, 
\[
\mathbb{E}\langle \zeta_2, h\rangle = n^{1/2} \bar{h}
(q_2 - 2q_1)^2 p_0^2 (1 - p_0)(q_1+q_2-2p_0 q_1)
\]
(where $\bar{h}$ is the
arithmetic mean of $h \in \mathbb{R}^n$), which will often diverge as $n \to \infty$. 

While the quality of the normal approximation in the short-term
has been established for a large class of occupancy processes, 
it is the approximation of the \emph{long-term
behaviour} of the process that is often more useful in the
population sciences. 
For this we require time-homogeneity, and, for each $n$, 
the deterministic process converges to some fixed point as $ t \to \infty$ (a set of convenient monotonicity conditions which implies this is provided in Theorem~\ref{thm:SmithStability}).
Under readily verifiable conditions on the global rule,
a time-homogeneous occupancy process
centered about its deterministic approximation,
converges in $\bigO(\log n)$ time to an approximately normal equilibrium (Corollary~\ref{cor:LongTerm}). 
\begin{corollary}
\label{cor:LongTerm}
Assume the conditions of Corollary~\ref{cor:CentralLim} hold, and
that \linebreak $\limsup_n \|DP^\n(\bv{p}_{\infty}^\n)\|_1 < 1$, 
and
$\sup_n \|\bv{p}_t^\n - \bv{p}_{\infty}^\n\|_{\infty} \to 0$
as $t \to \infty$. Then there exists a constant $c > 0$ independent
of $n$ such that for any sequence $\tau_n \leq c \log n$ with
$\tau_n \to \infty$ as $n\to\infty$,
\[
\langle \zeta_{\tau_n}^\n, h \rangle \convd \mathcal{N}(0,\mathcal{V}_{\infty}^2[h]),
\]
where $\mathcal{V}_{\infty}^2[h] = \lim_{t\to \infty} \mathcal{V}_t[h,h]$.
\end{corollary}
Unfortunately, we have not been able to prove a central limit theorem for
the process centered about the deterministic equilibrium for
arbitrary initial values of the process. Such a result would require
much stronger assumptions on the rate of convergence of the deterministic
approximation to equilibrium than we impose. 


The remainder of this document is structured as follows.
First, in \S\ref{sec:Coupling}, as a critical precursor to Theorem
\ref{thm:WasserBound}, the quality of the deterministic approximation (\ref{eq:Detalternative}) is considered, culminating in a functional error
bound in Theorem \ref{thm:LqrErrorBounds}. An important consequence of this
is a general uniform law of large numbers (Corollary \ref{cor:LawLargeNums}),
which we consider interesting in its own right. Due to its versatility, we choose to express our result in terms of the Rademacher complexity, defined by
\[
\Rad(\mathcal{H}) = \mathbb{E}\sup_{h \in \mathcal{H}} \frac1n
\sum_{i=1}^n h_i \sigma_i,
\]
where $\sigma_1,\dots,\sigma_n$ are independent Rademacher random variables $[\mathbb{P}(\sigma_i = 1) = \mathbb{P}(\sigma_i = -1) = \frac12]$. 
Rademacher complexity relates to other measures of the size of $\mathcal{H}$ as follows. 
Let $N_r(\mathcal{H})$ denote the $r$-covering number of $\mathcal{H}$, that is, the smallest number of balls of radius $r$ whose union contains $\mathcal{H}$. 
Then~\citep[Lemma 27.4]{Shalev:2014aa},
there exist universal constants $c_1,c_2 > 0$ such that 
\begin{equation}
\label{eq:RadBounds}
\Rad(\mathcal{H}) \leq \frac{c_1}{\sqrt{n}}\int_0^H \sqrt{\log N_r(\mathcal{H})}
\mathd r \leq c_2 H
\sqrt{\frac{\log |\mathcal{H}|}{n}},
\end{equation}
where $H = \sup_{h \in \mathcal{H}} \|h\|_{\infty}<\infty$. 
Furthermore,
if $\mathcal{H}$ is a set of binary vectors, and writing
$\mathscr{V}(\mathcal{H})$ as the Vapnik-Chervonenkis dimension of $\mathcal{H}$
as a set of functions from $\{1,\dots,n\}$ into $\{0,1\}$ (see
\citep[Definition 6.5]{Shalev:2014aa}), then \citep[Theorem 1]{Haussler:1995aa}
there exists $c_3 > 0$ such that 
\[
\Rad(\mathcal{H}) \leq c_3 \sqrt{\frac{\mathscr{V}(H)}{n}}.
\]
For more details on the Rademacher complexity, we direct the reader to \citep[\S 26,27]{Shalev:2014aa}.
Now, consider a sequence of occupancy processes $\{\bv{X}_t^\n\}_{n=1}^{\infty}$ indexed by number of
nodes with corresponding global rules $\{P_t^\n\}_{n=1}^{\infty}$, 
and sequences
$\{\alpha_t^\n\}_{n=1}^{\infty}$ and $\{\psi_t^\n\}_{n=1}^{\infty}$ for each integer $t \geq 0$. 

\begin{corollary}
\label{cor:LawLargeNums}
Consider a sequence of subsets $ \mathcal{H}_{n} \subset \mathbb{R}^{n} $ such that $ \mathrm{Rad}(\mathcal{H}_{n}) \to 0 $ and $ \sup_{n} \sup_{h\in\mathcal{H}_{n}} \|h\|_{\infty} < \infty $. Suppose that $ \sup_{n} \alpha_{s}^{(n)} < \infty $ for all $ s\leq t $. Then, as $ n \to \infty $,
\begin{enumerate}[label=(\alph*)]
\item \label{enu:LLNConvP} $\sup_{h \in \mathcal{H}_{n}} n^{-1} \sum_{i=1}^n h_i (X_{i,t}^\n - p_{i,t}^\n) \convp 0$ if $\psi_s^\n \to 0$ for all $s \leq t$;
\item \label{enu:LLNConvAS} $ \sup_{h \in \mathcal{H}_{n}} n^{-1} \sum_{i=1}^n h_i (X_{i,t}^\n - p_{i,t}^\n) \convas 0$ if $\{\psi_s^\n\}_{n=1}^{\infty} \in \ell^q$ for all $s\leq t$ and some $q$.
\end{enumerate}
\end{corollary}
In particular, the sequence $\{h_n\}_{n=1}^{\infty}$ of singletons
$h_n \in \mathbb{R}^n$ satisfies the conditions of Corollary~\ref{cor:LawLargeNums},
provided that $\sup_n \|h_n\|_{\infty} < \infty$.

Next, in  \S\ref{sec:Stein}, we outline the proofs of our main results (Theorem \ref{thm:WasserBound} and Corollary \ref{cor:CentralLim}) using Stein's method and the estimates from \S\ref{sec:Coupling}. The long-term behaviours of an occupancy
process and its approximations are considered
in \S\ref{sec:LongTerm}, culminating in a proof of Corollary~\ref{cor:LongTerm}.
A rate of convergence for Corollary~\ref{cor:LongTerm} is also provided.
Finally, in \S\ref{sec:Examples}, we discuss implications of our main results
in the context of spreading processes from epidemiology (Example~\ref{ex:Spreading}), Hanski's incidence
function model from population ecology (Example~\ref{ex:Hanski}), and dynamic random graph models (Example \ref{ex:RandomGraph}). Furthermore, Example~\ref{ex:Hanski}
demonstrates the implications of Corollary~\ref{cor:LawLargeNums} on the
convergence of empirical random measures associated with an occupancy process,
and Example~\ref{ex:RandomGraph}, in Proposition~\ref{prop:HomoCLT}, shows
how our main results, in conjunction with the approximation (\ref{eq:LinApproxGen}) and Proposition~\ref{prop:LinEstimate}, lead to
normal approximations for non-linear functions of an occupancy process.

\section{\label{sec:Coupling}Quality of the deterministic approximation}

Our first step is to discern when the macroscopic dynamics of the deterministic system reflect those of the stochastic model in some reasonable sense.
Limit theorems under a variety of mean field assumptions have been known for many years; see for example \citep{LeBoudec:2007aa}. More recently, \citep{Barbour:2015aa} considered the problem in a more general framework,  where concentration inequalities were obtained for the approximation error between empirical measures of $\bv{X}_t$ and $\bv{p}_t$ using the method of bounded differences. Their strategy follows that of \citep{Barbour:2008aa} by coupling the occupancy process together with another occupancy process, whose nodes transition \emph{independently}. 
By controlling higher-order moments, we extend their approach to cover the general case and improve on their findings, forming the foundation for the rest of this work.

Let $1\leq q,r \leq \infty$. For a vector $h \in \mathbb{R}^n$, let $\|h\|_q$ denote the traditional  $\ell^q$ norm; for a matrix $A$, $\|A\|_q$ the induced $\ell^q$ matrix norm. For a random variable $X$, $\|X\|_q = (\mathbb{E}|X|^q)^{1/q}$  is the $L^q$ norm, and similarly, for any function $f$ on $\mathbb{R}^d$, $\|f\|_q$ is the $L^q$ norm under Lebesgue measure. For a random vector~$\bv{X}$, we denote the $L^{q,r}$ norm by $\|\bv{X}\|_{q,r} = [\mathbb{E}(\sum_i |X_i|^q)^{r/q}]^{1/r} $. Next, we define the maximal $L^{q,r}$ norm acting on matrices $A = (a_{ij})$ by
\begin{equation}
\label{eq:MaxNormDef}
\|A\|_{q,r} = \begin{cases}
\left[\sum_{i=1}^m (\sum_{j=1}^n |a_{ij}|^q)^{r/q}\right]^{1/r} & \quad \mbox{ if $q \geq r$} \\
\left[\sum_{j=1}^n (\sum_{i=1}^m
|a_{ij}|^r)^{q/r}\right]^{1/q} & \quad \mbox{ if $r > q$,} 
\end{cases}
\end{equation}
with the obvious modifications for $q = \infty$ and $r = \infty$.  By a Minkowski type inequality \citep[Theorem 202]{Hardy:1952aa}, the norm  with the cases  in~(\ref{eq:MaxNormDef}) reversed does not exceed $\|A\|_{q,r}$. If $A$ is 
an $n \times n$ matrix, the special
cases $\|A\|_{2,1}$ and $\|A\|_{1,2}$ are bounded above by $\sqrt{n}
\|A\|_F$,
where $\|\cdot\|_F$ is the Frobenius (or Hilbert-Schmidt) norm. Finally, for notational convenience, for any matrix-valued function  $F = (f_{ij})$ and matrix norm $\|\cdot\|_M$, $\|F\|_M$ shall denote $\|(\|f_{ij}\|_{\infty})\|_M$.

Let $Df = (\partial_j f_i)_{ij}$ be the Jacobian matrix of a vector-valued differentiable function $f$ and let $D^{(2)} f = (\partial_j^2 f_i)_{ij}$ be the corresponding matrix of second derivatives. Our main result for this section
bounds the functional error between $\bv{X}_t$ and $\bv{p}_t$ under the $L^{q,r}$ norm.
\clearpage

\begin{theorem}
\label{thm:LqrErrorBounds}
For any $f = (f_1,\dots,f_m) \in \mathcal{C}^2(\mathbb{R}^n,\mathbb{R}^m)$, 
$q,r \geq 1$, and integers $t \geq 1$ and $s<t$,
\begin{multline}
\label{eq:LqrErrorBound}
\|f(\bv{X}_t)-f(\bv{p}_t)\|_{q,r} \leq 6\sqrt{\pi}\,nr^{3/2} \|Df\|_1 \sum_{s=0}^{t-1} \left(\frac1n + \psi_s\right) e^{4r\alpha_{s,t}}  \\
\qquad + \sqrt{\pi(q+r)} \|Df\|_{2,q} + \tfrac12 \|D^{(2)} f\|_{1,q}.
\end{multline}
\end{theorem}

To show Theorem \ref{thm:LqrErrorBounds}, as in \citep{Barbour:2008aa,Barbour:2015aa}, our approach for comparing
$\bv{X}_t$ and 
$\bv{p}_t$ is through a coupling with an intermediate occupancy process 
approximation $\bv{W}_t$, whose nodes evolve independently and satisfy 
$\mathbb{E}W_{i,t} = p_{i,t}$ for every $i=1,\dots,n$ and $t \geq 0$.
By measuring the total variation between $\bv{X}_t$ and $\bv{W}_t$,
the method of bounded differences (see \ref{sec:MomentIneq})
allows
for the approximation of functionals of $\bv{X}_t$
by the deterministic process $\bv{p}_t$.
Observe 
that the decomposition~(\ref{eq:TransDecomp}) implies that, for each 
$t \geq 0$, 
\[
X_{i,t+1} = X_{i,t} \ind\{U_{i,t}\leq S_{i,t}(\bv{X}_t)\} + (1 - X_{i,t}) 
\ind\{U_{i,t} \leq C_{i,t}(\bv{X}_t)\},
\]
where $(U_{i,t})_{i=1}^n$ is a collection of independent uniformly 
distributed random variables on $[0,1]$.
We construct $\bv{W}_t$, on the same probability space, by setting
$\bv{W}_0 = \bv{X}_0$ and
\[
W_{i,t+1} = W_{i,t} \ind\{U_{i,t}\leq S_{i,t}(\bv{p}_t)\} + (1 - W_{i,t}) 
\ind\{U_{i,t} \leq C_{i,t}(\bv{p}_t)\}.
\]
For each $i=1,\dots,n$ and $t \geq 1$, let $J_{i,t} := \max_{1\leq s\leq
t} \ind\{X_{i,s} \neq W_{i,s}\}$ and $\bar{J}_t := n^{-1} \sum_{i=1}^n
J_{i,t}$. Using the independence of each $W_{i,t}$ we obtain our first
approximation result.

\begin{lemma}
\label{lem:ApproxLemma1}
For any $f = (f_1,\dots,f_m) \in \mathcal{C}^2(\mathbb{R}^n,\mathbb{R}^m)$, 
$q,r \geq 1$, and integer $t \geq 1$,
\[
\|f(\bv{X}_t)-f(\bv{p}_t)\|_{q,r} \leq n \|Df\|_1 \| \bar{J}_t \|_r 
+ \sqrt{\pi (q+r)}\|D f\|_{2,q} + \tfrac12 \|D^{(2)} f\|_{1,q}. 
\label{eq:FuncBound}
\]
\end{lemma}
\begin{proof}
Define the function $w_q$ on $[0,1]^n$ by $w_q(\bv{x})^q = \sum_{i=1}^m |f_i(\bv{x})-\mathbb{E}f_i(\bv{W}_t)|^q$.
From the triangle inequality for the $L^{q,r}$ norm, $\|f(\bv{X}_t)-f(\bv{p}_t)\|_{q,r} \leq T_1+T_2+\|w_q(\bv{W}_t)\|_r$, where
$$
T_1 = \| f(\bv{X}_t) - f(\bv{W}_t) \|_{q,r} 
\qquad
\text{and}
\qquad
T_2 = \| \mathbb{E} f(\bv{W}_t) - f(\bv{p}_t) \|_q.
$$
$T_1$ is straightforward to bound using $J_{i,t}$, because
\[
\left(\sum_{i=1}^{m}
|f_i(\bv{X}_t) - f_i(\bv{W}_t)|^q\right)^{1/q} \leq {\sum_{i=1}^{m} \sum_{j=1}^{n}} 
\|\partial_j f_i\|_{\infty} J_{j,t}
\]
implies that $T_1 \leq n\|D f\|_1\|\bar{J}_t\|_r$. By the Lindeberg argument
(\ref{eq:Lindeberg}), $T_2 \leq \frac12 \|D^{(2)} f\|_{1,q}$. Using Jensen's inequality, 
\[
[\mathbb{E}w_q(\bv{W}_t)]^q \leq \sum_{i=1}^m \mathbb{E}|f_i(\bv{W}_t)-\mathbb{E}f_i(\bv{W}_t)|^q,
\]
so Theorem~\ref{thm:BernFuncMoments} implies $\mathbb{E}w_q(\bv{W}_t) \leq \sqrt{\frac{\pi q}{2}} \|Df\|_{2,q}$.
By the reverse triangle inequality, for each $j=1,\dots,n$,
\[
|\Delta_j w_q(\bv{x})| \leq \left(\sum_{i=1}^m |\Delta_j 
f_i(\bv{x})|^q\right)^{1/q}
\leq \left(\sum_{i=1}^m \|\partial_j f_i\|_{\infty}^q\right)^{1/q},
\]
and so another application of Theorem~\ref{thm:BernFuncMoments} gives $\|w_q(\bv{W}_t)\|_r \leq \sqrt{\frac\pi2}(q^{1/2} + r^{1/2})\|Df\|_{2,q}$. We conclude the proof with Cauchy's inequality: $q^{1/2} + r^{1/2} \leq  \sqrt{2(q+r)}$.
\end{proof}

The problem has now been reduced to obtaining appropriate bounds on the  moments of $\bar{J}_t$. By construction, for each $t=0,1,\dots$, 
\begin{align}
\label{eq:JRecurrence}
J_{i,t+1} \leq J_{i,t} &+ |\ind\{U_{i,t}\leq C_{i,t}(\bv{X}_t)\} - \ind\{U_{i,t}
\leq C_{i,t}(\bv{p}_t)\}|X_{i,t} \\
&+ |\ind\{U_{i,t}\leq S_{i,t}(\bv{X}_t)\} - \ind\{U_{i,t}\leq S_{i,t}
(\bv{p}_t)\}|(1-X_{i,t}), \nonumber
\end{align}
and so $J_{i,t+1}-J_{i,t}$ may be bounded above by the sum of two conditionally independent $\{0,1\}$-valued random variables. Thus, we have the following lemma which, together with Lemma \ref{lem:ApproxLemma1}, implies Theorem~\ref{thm:LqrErrorBounds}.
\begin{lemma}
\label{lem:ApproxLemma2} For any $q \geq 1$ and $t \geq 1$,
\begin{equation}
\|\bar{J}_t\|_q \leq 2q \sum_{s=0}^{t-1}(2n^{-1} + 3\beta_s \sqrt{\pi q} +
\gamma_s) e^{4q\alpha_{s,t}}.
\end{equation}
\end{lemma}
\begin{proof}
Letting $S_t = (S_{i,t})_{i=1}^n$ and 
$C_t = (C_{i,t})_{i=1}^n$, (\ref{eq:JRecurrence}) together with Lemma 
\ref{lem:PoisBinMoments} implies that, for any integer $t \geq 0$,
\[
\|\bar{J}_{t+1}\|_q \leq \|\bar{J}_t\|_q + \frac{2q}{n}\left\lbrace
2 + \|S_t(\bv{X}_t)-S_t(\bv{p}_t)\|_{1,q} + \|C_t(\bv{X}_t)-C_t(\bv{p}_t)
\|_{1,q}\right\rbrace.
\]
Applying Lemma \ref{lem:ApproxLemma1}, 
\[
\|\bar{J}_{t+1}\|_q \leq (1+4q\alpha_t)\|\bar{J}_t\|_q + 2q(2n^{-1}+
2\beta_t \sqrt{2 \pi q} + \gamma_t),
\]
and, since $1+4q\alpha_t \leq e^{4q\alpha_t}$, the lemma 
follows.
\end{proof}

Theorem~\ref{thm:LqrErrorBounds} is sufficient for proving our main results.
However, it seems prudent to examine its asymptotics in the number of
nodes and, in particular, demonstrate the law of large numbers result seen in Corollary \ref{cor:LawLargeNums}. 
For any vector $\bv{x}$ of length $n$, we let $\bar{\bv{x}} = n^{-1} \bv{x}$, so
that $\langle \bar{\bv{x}}, h \rangle = n^{-1} \sum_{i=1}^n h_i x_i$ becomes
a weighted average of the components of $\bv{x}$, appropriately normalised to
remain bounded as $n \to \infty$ for bounded $h$. 
For fixed $h \in \mathbb{R}^n$, by directly applying Theorem~\ref{thm:LqrErrorBounds} to $\langle \bar{\bv{X}}_t, h\rangle$ and $\langle \bar{\bv{p}}_t, h \rangle$, it is found that the variation (or, indeed, any one of the higher-order moments) of the empirical
measure of the occupancy process decays with order $\bigO(\sum_{s<t} \psi_t \vee n^{-1/2})$.
Proceeding further in this direction, the proof of Theorem~\ref{thm:LqrErrorBounds} implies a general log-normal concentration inequality (Corollary~\ref{cor:Concentration}) on the maximal deviation between $\langle \bar{\bv{X}}_t, h \rangle$ and $\langle \bar{\bv{p}}_t, h \rangle$ over $h \in \mathcal{H} \subset \mathbb{R}^n$, extending the result of \citep{Barbour:2015aa}. Furthermore, Corollary~\ref{cor:LawLargeNums} is an immediate consequence of Corollary~\ref{cor:Concentration}.
\begin{corollary}
\label{cor:Concentration}
Let $\mathcal{H} \subset \mathbb{R}^n$  with $H = \sup_{h \in
\mathcal{H}} \|h\|_{\infty}$. Then, for each $t \geq 0$, denoting $\Psi_t = 12\sqrt{\pi} (n^{-1} + \max_{s \leq t} \psi_s)$, for any $x > 1$,
\begin{multline}
\label{eq:Concentration}
\mathbb{P}\left(\sup_{h\in\mathcal{H}}|\langle \bar{\bv{X}}_t-\bar{\bv{p}}_t, h \rangle|> H t \Psi_{t} x+\Rad(\mathcal{H})\right) \\
\leq e^{-\tfrac12 nt^2 x^2 \Psi_t^2} + \sum_{s=1}^t 
\exp\left[-\frac{4 \alpha_{0,s} (\log x)^2}{(1+4\alpha_{0,t})^2} + 4\alpha_{0,s}\right].
\end{multline}
\end{corollary}
\begin{proof}
We first show 
\begin{equation}
\label{eq:RademacherBound}
\mathbb{E} \sup_{h\in\mathcal{H}} |\langle \bar{\bv{W}}_t- \bar{\bv{p}}_t, h \rangle|
\leq \Rad(\mathcal{H}).
\end{equation}
Since $\mathbb{E}\langle \bar{\bv{W}}_t, h \rangle = \langle \bar{\bv{p}}_t, h \rangle$, proceeding via the symmetrisation
method
\citep[Lemma 26.2]{Shalev:2014aa},
\[
\mathbb{E}\sup_{h\in\mathcal{H}}|\langle \bar{\bv{W}}_t - \bar{\bv{p}}_t, h \rangle| \leq
\mathbb{E}\sup_{h\in\mathcal{H}}\frac1n \sum_{i=1}^n h_i (W_{i,t}-\tilde{W}_{i,t}),
\]
where $\tilde{W}_{i,t}$ is an independent copy of $W_{i,t}$ for each
$i=1,\dots,n$. Now $ W_{i,t} - \tilde{W}_{i,t} \eqdist \sigma_{i} Z_{i} $ where $ \sigma_{i} $ are independent
Rademacher random variables and $ Z_{i} $ are $\{0,1\}$-valued random variables with $ \mathbb{P} \left( Z_{i} = 1\right) = 2 p_{i,t}(1-p_{i,t}) $, independent of $ \sigma_{i} $. So 
\[
\mathbb{E} \sup_{h\in\mathcal{H}} |\langle \bar{\bv{W}}_t-\bar{\bv{p}}_t, h \rangle|   \leq \mathbb{E} \sup_{h\in\mathcal{H}} \frac1n \sum_{i=1}^{n} h_{i} \sigma_{i} Z_{i} 
\]
The relation (\ref{eq:RademacherBound}) follows upon conditioning on each $Z_i$; indeed, for
any binary vector $\bv{z}=(z_1,\dots,z_n) \in \{0,1\}^n$,
\[
\mathbb{E}\sup_{h\in\mathcal{H}}\frac1n\sum_{i=1}^n h_i\sigma_i z_i
= \mathbb{E}\sup_{h\in\mathcal{H}}\mathbb{E}\left[\cond{\frac1n
\sum_{i=1}^n h_i \sigma_i}\sigma_i:\,z_i=1\right] \leq \Rad(\mathcal{H}).
\]
As a consequence of the one-sided McDiarmid inequality (\ref{eq:McDiarmid}),
\begin{multline}
\label{eq:McDiarmidAppl}
\mathbb{P}\left(\sup_{h\in \mathcal{H}} |\langle \bar{\bv{W}}_t - \bar{\bv{p}}_t, h \rangle| >
\tfrac12 Ht\Psi_t x + \Rad(\mathcal{H})\right) \\
\leq \exp\left(-\tfrac12 nt^2 x^2 \Psi_t^2
\right).
\end{multline}
Now let $J_t(\mathcal{H}) = \sup_{h \in \mathcal{H}}|\langle \bar{\bv{X}}_t - \bar{\bv{W}}_t, h \rangle|$, so that $\mathbb{E}J_t(\mathcal{H})^q \leq
H^q \|\bar{J}_t\|_q^q$.
Lemma \ref{lem:ApproxLemma2} provides a partial estimate for the moment-generating function of $J_t(\mathcal{H})$:
for any $q \geq 1$, denoting $L_s = \log(4n^{-1}+6\beta_s\sqrt{\pi} + 2\gamma_s) + \log(H t)$ for each
$s \leq t$,
\[
\mathbb{E}e^{q \log J_t(\mathcal{H})} \leq \sum_{s=0}^{t-1} \exp(4q^2 \alpha_{s,t} + q L_s + q \log q).
\]
Assuming that $\log x \geq 1 + 4 \alpha_{0,t}$, choose $q = \log x (1 +
4 \alpha_{0,t})^{-1}$, so, by the Chernoff approach,
\begin{multline}
\label{eq:ChernoffBound}
\mathbb{P}\left(\log J_t(\mathcal{H}) > \log x + \max_{s\leq t} L_s\right) \\
\leq \mathbb{E}\exp\left(q \log J_t(\mathcal{H}) - q \max_{s \leq t} L_s - q\log x\right)
\leq \sum_{s=1}^t \exp\left(-\frac{4\alpha_{0,s} (\log x)^2}{(1+4\alpha_{0,t})^2}\right).
\end{multline}
To extend to $0 \leq \log x < 1 + 4\alpha_{0,t}$, it suffices to add $4 \alpha_{0,s}$
into each exponent, whereupon inequality (\ref{eq:ChernoffBound}) becomes the
trivial bound. Together with (\ref{eq:McDiarmidAppl}), this implies 
Corollary~\ref{cor:Concentration}.
\end{proof}

We conclude this section with an error estimate for approximating
$f(\bv{X}_t)$ for arbitrary three-times differentiable functions $f$, by the
linear approximation (\ref{eq:LinApproxGen}).
Aside from acting as a fundamental component of the proof of 
Proposition~\ref{prop:OneStep} below, 
it extends Theorem~\ref{thm:WasserBound} to
provide error estimates for the normal approximation to non-linear functions
of $\bv{X}_t$. 
\begin{proposition}
\label{prop:LinEstimate}
There is a universal constant $C > 0$ such that,
for any $f \in \mathcal{C}^3([0,1]^n)$ and $t \geq 1$,
\begin{multline*}
\mathbb{E}| f(\bv{X}_t) - f(\bv{p}_t) - \langle\zeta_t, \sqrt{n}\nabla f(\bv{p}_t) 
\rangle
| \\
\leq C{\sqrt{1 + \log n}} \left[1 + \textstyle{\sum\limits_{s=0}^{t-1}} (n^{-1} + n
\psi_s^2) t e^{16 \alpha_{s,t}}\right]\\
\left[ n \max_{j,k=1,\dots,n}
\|\partial_j \partial_k f\|_{\infty} +\sqrt{n}\max_{j=1,\dots,n}\textstyle{
\sum\limits_{k=1}^n \|\partial_j \partial_k^2 f\|_{\infty}}
\right].
\end{multline*}
\end{proposition}
\begin{proof}
Relying on the independent node approximation once again, the proof follows
by comparing $\bv{X}_t$ to $\bv{W}_t$, whence $\bv{W}_t$ may be
compared to $\bv{p}_t$ using Lemma \ref{lem:LinMOBD}. Firstly,
\[
\mathbb{E}\left|f(\bv X_{t})-f(\bv W_{t})-\textstyle{\sum\limits_{j=1}^{n}\partial_{j}f(\bv W_{t})(X_{j,t}-W_{j,t})}\right| \leq n^2 \max_{1\leq j,k \leq n} \|\partial_j \partial_k f\|_{\infty} \|\bar{J}_t\|_2^2,
\]
which may be controlled by Lemma \ref{lem:ApproxLemma2}. Focusing on the
remaining cross-term, by two applications of H\"{o}lder's inequality,
\begin{multline*}
\mathbb{E}\left|\textstyle{\sum\limits_{j=1}^{n}}[\partial_{j}f(\bv W_{t})-\partial_{j}f(\bv p_{t})](X_{j,t}-W_{j,t})\right| \\
\leq\left(\mathbb{E}\max_{j=1,\dots,n}[\partial_{j}f(\bv W_{t})-\partial_{j}f(\bv p_{t})]^{2}\right)^{1/2}\cdot n\|\bar{J}_{t}\| _{2},
\end{multline*}
and consequently, due to (\ref{eq:MaxOrlicz}), it suffices to estimate
$\|\partial_j f(\bv{W}_t)-\partial_j f(\bv{p}_t)\|_{\Psi}$. But, by
Theorem~\ref{thm:BernFuncMoments} and~(\ref{eq:Lindeberg}),
\[
\left\lVert \partial_{j}f\left(\bv W_{t}\right)-\partial_{j}f\left(\bv p_{t}\right)\right\rVert _{\Psi}
\leq C\sqrt{n}\max_{k=1,\dots,n}\left\lVert \partial_j \partial_kf\right\rVert _{\infty}+\frac{1}{2}\sum_{k=1}^{n}\left\lVert \partial_{j}\partial_{k}^{2}f\right\rVert _{\infty},
\]
and the result follows.
\end{proof}

\section{\label{sec:Stein}Stein's method for occupancy processes}

We now turn to the problem of normal approximation, and proving 
Theorem~\ref{thm:WasserBound}. Clearly, a direct application of the independent
node coupling $\bv{W}_t$ will not suffice. However, the conditional
independence property of occupancy processes immediately implies a
conditional central limit result: conditioning on $\bv{X}_t$,
$\langle\zeta_{t+1},h\rangle$ converges to a normal random variable for any $h\in
\mathbb{R}^n$. An estimate of the convergence rate is given by the
classical Berry-Esseen bound, for which some of the simplest proofs make
use of Stein's method. To obtain the required
\emph{unconditional} estimate is more challenging, but fundamentally
relies on this property.

The idea behind Stein's method is to estimate the difference
between the expectations $\mathbb{E}g(X)$ and $\mathbb{E}g(Z)$ through a characterising operator (often called the 
\emph{Stein operator}) $\mathcal{A}$ which has the following property: if a random variable $X$ satisfies $\mathbb{E}\mathcal{A}f(X) = 0$ for all $f$ in an appropriate class of functions, then $X \eqdist Z$. In the case of normal approximation, where 
$Z \sim \mathcal{N}(\mu,\sigma^2)$, the operator
\[
\mathcal{A}f(x) = \sigma^2 f'(x) - (x - \mu) f(x)
\]
suffices. Stein recognized that if, for some chosen function $g$, $f_g$ solves the \emph{Stein equation}
\[
\mathcal{A}f_g(x) = g(x) - \mathbb{E}g(Z),
\]
then, provided $X$ is similar in law to $Z$, $\mathcal{A}f_g(X)$ should also have small expectation. Stein's method is often successful because bounding $\mathbb{E}\mathcal{A}f(X)$ is an appreciably simpler task to perform in general. In particular, the method is known to be remarkably flexible for handling sums of random variables  with complex dependencies, and is well-suited for our purposes. See \citep{Chen:2010aa} for
a comprehensive exposition of normal approximation techniques involving Stein's method. The application of Stein's method to the one-point distributions of
a discrete-time process was considered by Goldstein
\citep{Goldstein:2004aa} to develop normal approximations for
\emph{hierarchical structures}. While the contraction principle used in
his analysis does not apply here, the one-step linearisation approach in
our analysis is of a similar flavour. 
Indeed, we may divine the relationship
between $\zeta_t$ and $\xi_t$ by way of a series of one-step approximations.
For any time $t\geq 1$, consider the $s$-step normal approximation
$\zeta_t^{(s)}$ defined for $h \in \mathbb{R}^n$ by
\[
\langle \zeta_t^{(s)}, h \rangle =\langle \zeta_{t-s}, D_{t-s,t} h \rangle  +
\sum_{r=t-s+1}^t \sigma_r[D_{r,t} h] \cdot z_r,
\]
where each $z_r$ is an independent standard normal random variable. We
denote the special case $s = 1$ by $\tilde{\zeta}_t$ and remark that
$\zeta_t^{(t)}$ and $\xi_t$ are equal in distribution. By working with
the $L^q$ metric between distributions of random variables, we have at
our disposal the following contraction identity under translations by
independent random variables. For any random variables $X, Y,$ and $Z$,
such that $Z$ is independent of $X$ and $Y$,
\begin{equation}
\label{eq:Contraction}
\|\law(X+Z)-\law(Y+Z)\|_q \leq \|\law(X)-\law(Y)\|_q,
\end{equation}
which is just a restatement of Young's inequality for convolutions.
Together with the triangle inequality,
\begin{align*}
\|\law \langle \zeta_t, h \rangle-\law \langle \xi_t, h \rangle\|_q
&\leq \sum_{s=0}^{t-1}
\|\law \langle \zeta_t^{(s)}, h \rangle  - \law \langle \zeta_t^{(s+1)}, h\rangle\|_q,\\
&\leq \sum_{s=1}^t
\|\law \langle \zeta_s, D_{s,t}h\rangle-\law \langle \tilde{\zeta}_s, D_{s,t}h\rangle\|_q.
\end{align*}
Indeed, for any time $t$,
\[
\|D_t\|_{\infty} = \max_{1\leq j\leq n} \sum_{i=1}^n 
\|\partial_j P_{i,t}\|_{\infty} \leq 1 + \alpha_t,
\]
implying that $\|D_{s+1,t} h\|_{\infty} \leq \|h\|_{\infty}
\exp(\alpha_{s,t})$, and hence Theorem~\ref{thm:WasserBound} follows
from Proposition~\ref{prop:OneStep}. Here and throughout, $C$ will
denote a universal constant, but it will not necessarily be the same on
each appearance. Furthermore, it will be assumed implicitly throughout
that $\sigma_t[h] > 0$ for each $t \geq 1$.
\begin{proposition}
\label{prop:OneStep}
For any integer $t \geq 0$,
\begin{align*}
\|\law \langle \zeta_{t+1}, h \rangle -\law \langle \tilde{\zeta}_{t+1}, h \rangle \|_q
\leq 
C\sqrt{\frac{1+\log n}{n}} \cdot \frac{\|h\|_{\infty}^{4-1/q} \,\kappa_t }{\sigma_{t+1}^{4-2/q}[h]}.
\end{align*}
\end{proposition}
The strategy of proof is surprisingly
simple. First, using the conditional independence property, apply
Stein's method under the appropriate conditional probability space to
obtain an estimate for the one-step normal approximation. For this, let
$\mathbb{E}_t$, $\mathbb{P}_t$, and $\var_t$ denote expectation,
probability, and variance conditional on $\bv{X}_t$. Let $h \in
\mathbb{R}^n$ be arbitrary, and for any $f \in
\mathcal{C}^1(\mathbb{R})$, define $\mathcal{A}_t f$ and
$\tilde{\mathcal{A}}_t f$ as the Stein operators
\begin{align*}
\mathcal{A}_t f(x) &= \sigma_{t+1}^2[h]\cdot f'(x) -
\{x - \langle \zeta_t, D_t h \rangle \} f(x), \\
\tilde{\mathcal{A}}_t f(x) &= \sigma_{t+1}^2[h] \cdot f'(x) -
\{x - \mathbb{E}_t\langle \zeta_{t+1}, h \rangle \} f(x).
\end{align*}
We proceed by estimating $S_t f = \mathcal{A}_t f(\langle \zeta_{t+1}, h \rangle) $
through $\tilde{S}_t f = \tilde{\mathcal{A}}_t f(\langle \zeta_{t+1}, h \rangle) $. As is
customary with Stein's method, we first focus on the second term in
these expressions. Indeed,
\[
\mathbb{E}_{t}[\langle \zeta_{t+1}, h \rangle f(\langle \zeta_{t+1}, h \rangle) ]=
\frac{1}{\sqrt{n}}\sum_{i=1}^{n}h_{i}\{ \mathbb{E}_{t}[X_{i,t+1}f(\langle \zeta_{t+1}, h \rangle) ]
-p_{i,t+1}\mathbb{E}_{t}f(\langle \zeta_{t+1}, h \rangle) \},
\]
and hence, by conditioning on the event $X_{i,t+1} = 1$,
\[
\mathbb{E}_{t}[X_{i,t+1}f(\langle \zeta_{t+1}, h \rangle) ]=
\mathbb{E}_{t}f\left[\langle \zeta_{t+1}, h \rangle +\frac{h_{i}}{\sqrt{n}}(1-X_{i,t+1})\right]
P_{i,t} (\bv X_{t}).
\]
Conditioning further on $X_{i,t+1} = 0$ reveals that
\begin{multline}
\label{eq:SteinMain}
\mathbb{E}_t[\{\langle \zeta_{t+1}, h \rangle  -\mathbb{E}_t\langle \zeta_{t+1}, h \rangle  \}
f(\langle \zeta_{t+1}, h \rangle) ] \\= \frac1n \sum_{i=1}^n h_i^2
P_{i,t}(\bv{X}_t)[1-P_{i,t}(\bv{X}_t)] \mathbb{E}_t \int_0^1
f'\left[\langle \zeta_{t+1}, h \rangle  + \frac{h_i}{\sqrt{n}} (u - X_{i,t+1}) \right]
\mathd u.
\end{multline}
The method of bounding the integral term, and therefore $\mathbb{E}_t
\tilde{S}_t f$, varies depending on $q$. Denoting by $\phi$ and $\Phi$
the density and distribution functions, respectively, of a standard
normal random variable, we restate the fundamental lemma for Stein's
method for normal approximation, linking bounds on the Stein equations
with bounds for the Wasserstein and Kolmogorov metrics (the proof of
which may be found in \citep[Lemmas 2.3 \& 2.4]{Chen:2010aa}).
\begin{lemma}[\textsc{Stein's Lemma}]
\label{lem:Stein}
For any $\mu \in \mathbb{R}$, $\sigma > 0$, and
$g\in\mathcal{C}^{1}(\mathbb{R})$, defining $f_{g}$ as the function
satisfying
\begin{equation}
\label{eq:SteinEq}
\sigma^2 f_g'(x)-(x-\mu)f_g(x)=g(x)-\frac1{\sigma}\int_\mathbb{R}
g(x) \phi\left(\frac{x-\mu}{\sigma}\right) \mathrm{d}x,
\end{equation}
we have that $f_{g}\in\mathcal{C}^{2}(\mathbb{R})$ and satisfies
\[
\|f_g\|_{\infty}\leq 2 \|g'\|_{\infty},
\qquad \|f_g'\|_{\infty}\leq\frac{1}{\sigma}\|g'\|_{\infty},
\qquad \|f_g''\|_{\infty}\leq\frac{2}{\sigma^{2}}\|g'\|_{\infty}.
\]
Alternatively, for any $z\in\mathbb{R}$, defining $f_{z}$ as the
function satisfying
\[
\sigma^2 f_z'(x)-(x-\mu)f_{z}(x)=\ind_{x\leq z}-
\Phi\left(\frac{x-\mu}{\sigma}\right),
\]
we have that $f_z\in\mathcal{C}^{1}(\mathbb{R})$ satisfying
$\|f_z\|_{\infty}\leq\sigma^{-1}$, $\|f_z'\|_{\infty}\leq\sigma^{-1}$
and, for any $h \in \mathbb{R}$,
\[
|(x+h-\mu)f_z(x+h)-(x-\mu)f_{z}(x)|\leq\frac{1}{\sigma}\left(\frac{|x-\mu|}{\sigma}
+\frac{\sqrt{2\pi}}{4}\right)|h|.
\]
\end{lemma}
The Wasserstein ($q = 1$) case is outlined in \S\ref{sec:Wasserstein},
while the more difficult Kolmogorov ($q = \infty$) case is treated in
\S\ref{sec:Kolmogorov}. The estimate for general $q$ follows from these
two cases by interpolation. Simultaneously, by utilising the independent
node coupling from \S\ref{sec:Coupling}, the remaining dependence on
$\bv{X}_t$ may be removed, and a bound on $\mathbb{E} S_t f$ obtained.
Analogously to \citep[Proposition 4.1]{Goldstein:2004aa}, this involves
two key estimates, one approximating $\mathbb{E}_t\langle \zeta_{t+1}, h \rangle $ by
$\langle \zeta_t, D_t h \rangle $, and another for $\var_t\langle \zeta_{t+1}, h \rangle $ by
$\sigma_{t+1}^2[h]$. The former is a direct consequence of 
Proposition~\ref{prop:LinEstimate}:
\begin{equation}
\label{eq:MeanApprox}
\mathbb{E}|\mathbb{E}_t\langle \zeta_{t+1}, h \rangle -\langle \zeta_t, D_t h \rangle | \leq C \|h\|_{\infty}
\kappa_t n^{-1/2} \sqrt{1 + \log n} .
\end{equation}
For the latter, by defining $V_t(\bv{x})=n^{-1}\sum_{i=1}^n h_i^2
P_{i,t}(\bv{x})[1-P_{i,t}(\bv{x})]$ for $\bv{x} \in [0,1]^n$, we have
\[
\mathbb{E}|\sigma_{t+1}^2[h] - \var_t\langle \zeta_{t+1}, h \rangle |
= \mathbb{E}|V_t(\bv{X}_t)-V_t(\bv{p}_t)|.
\]
Computing the derivatives of $V_t$ reveals
\[
\|DV_t\|_1 \leq n^{-1} \|h\|_{\infty}^2 (1 + \alpha_t), \qquad
\|DV_t\|_{2,1} \leq n^{-1/2} \|h\|_{\infty}^2 (1 + \alpha_t),
\]
\[
\|D^{(2)} V_t\|_{1,1} \leq \|h\|_{\infty}^2 (\gamma_t + 2\beta_t^2).
\]
The desired estimate now follows by applying 
Theorem~\ref{thm:LqrErrorBounds}:
\begin{equation}
\label{eq:SigmaApprox}
\mathbb{E}|\sigma_{t+1}^2[h] - \var_t\langle \zeta_{t+1}, h \rangle |
\leq C \|h\|_{\infty}^2 \kappa_t n^{-1/2} .
\end{equation}

\begin{remark}
\label{rem:MeanField}
The $\log n$ term in Proposition~\ref{prop:OneStep} arises only from
equation (\ref{eq:MeanApprox}), and so may be removed, provided
one can derive an $\mathcal{O}(n^{-1/2})$ bound for this term.
For example, consider an occupancy process with
\begin{equation}
\label{eq:HanskiForm}
\textstyle{
S_{i,t}(\bv{x}) = f_i\left(n^{-1} \sum_{j=1}^n s_{ij} x_j\right),\quad
C_{i,t}(\bv{x}) = g_i\left(n^{-1} \sum_{j=1}^n c_{ij} x_j\right),}
\end{equation}
where $f_i,g_i \in \mathcal{C}^2([0,\infty))$ and $s_{ij},c_{ij} \geq 0$, 
with $s_{ii} = c_{ii} = 0$. This particular process was studied in
\citep{Barbour:2015aa}. 
Denote $\bv{s}_i = (s_{i1},\dots,s_{in})$ and similarly for $\bv{c}_i$. In this
instance, the Hessian matrices of $S_{i,t},C_{i,t}$ conveniently 
factorise into a sum of rank-one real-valued matrices scaled by real-valued functions,
and for any $j,k\neq i$,
\[
\partial_j\partial_k P_{i,t}(\bv{x}) = \frac{x_i s_{ij}s_{ik}}{n^2} f_i''\left(\frac1n\sum_{l=1}^n
s_{il}x_l\right) + \frac{(1-x_i)c_{ij}c_{ik}}{n^2} g_i''\left(\frac1n\sum_{l=1}^n
c_{il}x_l\right).
\]
So by Taylor's Theorem, 
\[
\mathbb{E}|\mathbb{E}_t\langle \zeta_{t+1}, h \rangle -\langle \zeta_t, D_t h \rangle |
\leq \frac{\|h\|_{\infty}}{n^{3/2}} \sum_{i=1}^n \|f_i''\|_{\infty}
\mathbb{E}\langle \zeta_t, \bv{s}_i\rangle^2 + \|g_i''\|_{\infty}\mathbb{E}\langle \zeta_t, \bv{c}_i\rangle^2,
\]
where here $\|f_i''\|_{\infty}$ is understood as the supremum of $f_i$ 
over the convex hull of
$\{n^{-1} \sum_{j=1}^n s_{ij}\}_{i=1}^n$, and likewise for $\|g_i''\|_{\infty}$.
Observing that
\[
\textstyle{n^{-1}\sum_{i=1}^n\|f_i''\|_{\infty}\|\bv{s}_i\|_{\infty}^2
+\|g_i''\|_{\infty}\|\bv{c}_i\|_{\infty}^2 \leq 2 n \Gamma_t},
\]
with the help of Theorem~\ref{thm:LqrErrorBounds}
it can be shown that
\begin{equation}
\label{eq:MeanApprox2}
\mathbb{E}|\mathbb{E}_t\langle \zeta_{t+1}, h \rangle -\langle \zeta_t, D_t h \rangle | \leq C \|h\|_{\infty}
\kappa_t n^{-1/2}.
\end{equation}
\end{remark}

\subsection{\label{sec:Wasserstein}The Wasserstein metric}
Let $g \in \mathcal{C}^1(\mathbb{R})$ be arbitrary, and take $f_g$ to be the 
solution to $\mathcal{A}_t f_g(x) = g(x) - \mathbb{E}_tg(\langle \tilde{\zeta}_{t+1}, h \rangle)$.
Since $f_g \in \mathcal{C}^2(\mathbb{R})$,
there exists a random variable $Y_{i,t}$ such that
\[
\int_0^1 f'\left[\langle \zeta_{t+1}, h \rangle 
+ \frac{h_i}{\sqrt{n}} (u- X_{i,t+1}) \right] \mathd u = f_g'(\langle \zeta_{t+1}, h \rangle) +
\frac{h_i (1 - 2X_{i,t+1})}{2\sqrt{n}} f_g''(Y_{i,t}) .
\]
Thus, $|\mathbb{E}_t \tilde{S}_t f| \leq T_1 + T_2$, where
\begin{align*}
T_1 &= |\sigma_{t+1}^2[h]-\var_t\langle \zeta_{t+1}, h \rangle ||\mathbb{E}_t
f_g'(\langle \zeta_{t+1}, h \rangle) |, \\
T_2 &= \frac{1}{2 n^{3/2}}\sum_{i=1}^n |h_i|^3 
P_{i,t}(\bv{X}_t)[1-P_{i,t}(\bv{X}_t)]|\mathbb{E}_t (1-2X_{i,t+1}) f_g''(Y_{i,t})|.
\end{align*}
Using~(\ref{eq:SigmaApprox}) to bound $T_1$ and Lemma \ref{lem:Stein} to 
bound the derivatives of $f_g$,
\[
\mathbb{E}|T_1| \leq \frac{C \|g'\|_{\infty}\|h\|_{\infty}^2 \kappa_t}{\sigma_{t+1}[h]\sqrt{n}} , \qquad
|T_2| \leq \frac{C \|g'\|_{\infty} \|h\|_{\infty}^3}{\sigma_{t+1}^2[h]\sqrt{n}} .
\]
Since $\sigma_{t+1}[h]\leq \|h\|_{\infty}$, it follows that
\[
|\mathbb{E} \tilde{S}_t f_g| \leq \mathbb{E}|\mathbb{E}_t\tilde{S}_t f_g| \leq
\frac{C \|g'\|_{\infty} \|h\|_{\infty}^3 \kappa_t}{\sigma_{t+1}^2[h] \sqrt{n}},
\]
from which (\ref{eq:MeanApprox}) implies 
\begin{equation}
\label{eq:WasserFinal}
|\mathbb{E}g(\langle \zeta_{t+1}, h \rangle)  - \mathbb{E}g(\langle \tilde{\zeta}_{t+1}, h \rangle) |
\leq C \sqrt{\frac{1+\log n}{n}} \cdot \frac{\|g'\|_{\infty} \|h\|_{\infty}^3 \kappa_t}{\sigma_{t+1}^2[h]},
\end{equation}
which is exactly Proposition~\ref{prop:OneStep} with $q = 1$.
%

\subsection{\label{sec:Kolmogorov}The Kolmogorov metric}
Let $z \in \mathbb{R}$ be arbitrary, and take $f_z$ to be the solution to
$\mathcal{A}_t f_z(x) = \ind\{x \leq z\} - \mathbb{P}_t(\langle \tilde{\zeta}_{t+1}, h \rangle 
\leq z)$. By rearranging the Stein equation for $f_z'$ and inserting into
(\ref{eq:SteinMain}), we obtain
\begin{multline}
\label{eq:KolmoMain}
\sum_{i=1}^n v_{i,t} \left[\int_0^1 \mathbb{P}_t\left(\langle \zeta_{t+1}, h \rangle  +
\frac{h_i}{\sqrt{n}} (u - X_{i,t+1}) \leq z\right) \mathd u
- \mathbb{P}_t(\langle \tilde{\zeta}_{t+1}, h \rangle  \leq z)\right] \\
= \sigma_{t+1}^2[h] T_1 + T_2 - \sum_{i=1}^n v_{i,t} \int_0^1 \mathbb{E}_t I_{i,t}(u) \mathd u
\end{multline}
where $v_{i,t} = n^{-1} h_i^2 P_{i,t}(\bv{X}_t)[1-P_{i,t}(\bv{X}_t)]$, 
and $T_1$, $T_2$, $I_{i,t}(u)$ for $u \in [0,1]$ are given by
\begin{align*}
T_1 &= [\langle \zeta_t, D_t h \rangle  - \mathbb{E}_t\langle \zeta_{t+1}, h \rangle ]\mathbb{E}_t 
f_z(\langle \zeta_{t+1}, h \rangle)  \\
T_2 &= [\sigma_{t+1}^2[h] - \var_t\langle \zeta_{t+1}, h \rangle ]
\mathbb{E}_t[\{\langle \zeta_{t+1}, h \rangle -\langle \zeta_t, D_t h \rangle \}f_z(\langle \zeta_{t+1}, h \rangle) ]
\end{align*}
\vspace{-2\baselineskip}
\begin{multline*}
I_{i,t}(u) = \left\{\langle \zeta_{t+1}, h \rangle +\frac{h_{i}}{\sqrt{n}}(u-X_{i,t+1})-\langle \zeta_t, D_t h \rangle \right\} \\
\times f_{z}\left(\langle \zeta_{t+1}, h \rangle +\frac{h_{i}}{\sqrt{n}}(u-X_{i,t+1})\right)\\
-\left\{\langle \zeta_{t+1}, h \rangle -\langle \zeta_t, D_t h \rangle \right\} f_{z}(\langle \zeta_{t+1}, h \rangle).
\end{multline*}
Applying Lemma \ref{lem:Stein} to $I_{i,t}$ implies 
\[
\mathbb{E}_t |I_{i,t}(u)| \leq \frac{1}{\sigma_{t+1}[h]}
\left(\frac{|\langle \zeta_{t+1}, h \rangle -\langle \zeta_t, D_t h \rangle |}{\sigma_{t+1}[h]}+\frac{\sqrt{2\pi}}{4}\right)
\frac{\|h\|_{\infty}}{\sqrt{n}}.
\]
Now, since $\|f_z\|_{\infty}$ is bounded uniformly in $z$, the right-hand
side of (\ref{eq:KolmoMain}) may be bounded in magnitude independently
of $z$. For the moment, let $M$ denote such a bound. 
Then, since
\begin{align*}
\mathbb{P}_{t}\left(\langle \zeta_{t+1}, h \rangle +\frac{h_{i}}{\sqrt{n}}u\leq z+
\frac{\|h\|_{\infty}}{\sqrt{n}}\right) 
&\geq \mathbb{P}_{t}(\langle \zeta_{t+1}, h \rangle \leq z), \\
\left|\mathbb{P}_{t}\left(\langle \tilde{\zeta}_{t+1}, h \rangle \leq z+
\frac{\|h\|_{\infty}}{\sqrt{n}}\right)-
\mathbb{P}_{t}(\langle \tilde{\zeta}_{t+1}, h \rangle \leq z)\right| 
&\leq \frac{\|h\|_{\infty}}{\sigma_{t+1}[h] \sqrt{2 \pi n}},
\end{align*}
substituting $z + n^{-1/2}\|h\|_{\infty}$ for $z$ in (\ref{eq:KolmoMain})
gives
\[
\var_t\langle \zeta_{t+1}, h \rangle [\mathbb{P}_t(\langle \zeta_{t+1}, h \rangle \leq z)
-\mathbb{P}_t(\langle \tilde{\zeta}_{t+1}, h \rangle \leq z)]
\leq M + \frac{C{\|h\|_{\infty}^3}}{\sigma_{t+1}[h] \sqrt{n}}.
\]
By performing a similar procedure for the lower bound,
\begin{equation}
\label{eq:KolmoMain2}
\sigma_{t+1}^2[h]|\mathbb{P}_t(\langle \zeta_{t+1}, h \rangle \leq z)-
\mathbb{P}_t(\langle \tilde{\zeta}_{t+1}, h \rangle \leq z)|
\leq M + \frac{C{\|h\|_{\infty}^3}}{\sigma_{t+1}[h] \sqrt{n}} + T_3,
\end{equation}
where
\[
T_3 = |\sigma_{t+1}^2[h] - \var_t\langle \zeta_{t+1}, h \rangle |
|\mathbb{P}_t(\langle \zeta_{t+1}, h \rangle \leq z) - \mathbb{P}_t(\langle \tilde{\zeta}_{t+1}, h\rangle
\leq z)|.
\]
Inequalities (\ref{eq:MeanApprox}) and (\ref{eq:SigmaApprox}) together with
the estimate for $\|f_z\|_{\infty}$ in Lemma \ref{lem:Stein} immediately imply
\[
\mathbb{E}|T_1| \leq \frac{C\|h\|_{\infty}\kappa_t}
{\sigma_{t+1}[h] \sqrt{n}}, \qquad
\mathbb{E}|T_3| \leq \frac{C\|h\|_{\infty}^2 \kappa_t}{\sqrt{n}}.
\]
Once again, since $\var_t\langle \zeta_{t+1}, h \rangle  \leq \|h\|_{\infty}^2$,
\[
\mathbb{E}_t|\langle \zeta_{t+1}, h \rangle -\langle \zeta_t, D_t h \rangle | \leq \|h\|_{\infty}
+ |\mathbb{E}_t\langle \zeta_{t+1}, h \rangle  - \langle \zeta_t, D_t h \rangle |,
\]
and by liberal use of (\ref{eq:MeanApprox}) and (\ref{eq:SigmaApprox}) with
Lemma \ref{lem:Stein},
\[
\mathbb{E}|T_2| \leq C\sqrt{\frac{1 + \log n}{n}} \cdot \frac{\|h\|_{\infty}^3
\kappa_t}{\sigma_{t+1}[h]},
\qquad
\mathbb{E}|I_{i,t}(u)| \leq 
\frac{C \|h\|_{\infty}^2 \kappa_t}{\sigma_{t+1}^2[h] \sqrt{n}}.
\]
Altogether, combining these estimates with (\ref{eq:KolmoMain}) and
(\ref{eq:KolmoMain2}) gives
\[
\|\law\langle \zeta_{t+1}, h \rangle -\law\langle \tilde{\zeta}_{t+1}, h \rangle \|_{\infty}
\leq C \sqrt{\frac{1+\log n}{n}} \cdot \frac{\|h\|_{\infty}^4 \kappa_t}{\sigma_{t+1}^4[h]},
\]
which is Proposition~\ref{prop:OneStep} with $q = \infty$. 

\subsection{A central limit theorem}
It now only remains to prove Corollary~\ref{cor:CentralLim}. 
To do so, we shall once again make use of the one-step 
approximations~$\tilde{\zeta}_t$. Unfortunately, 
Proposition~\ref{prop:OneStep} is not quite sufficient as $\sigma_t^{-2}[h]$
is potentially unbounded. Instead, by utilising 
Proposition~\ref{prop:OneStep} in conjunction with a crude bound, 
the dependence on
$\sigma_t[h]$ may be removed at the expense of 
a suboptimal exponent in $n$. For any function $g \in \mathcal{C}^1(\mathbb{R})$,
\begin{multline*}
|\mathbb{E}g(\langle \zeta_{t+1}, h \rangle) -\mathbb{E}g(\langle \tilde{\zeta}_{t+1}, h \rangle )|
\leq \|g'\|_{\infty}\{\mathbb{E}|\langle \zeta_{t+1}, h \rangle -\mathbb{E}_t\langle \zeta_{t+1}, h \rangle | + \\
\mathbb{E}|\mathbb{E}_t\langle \zeta_{t+1}, h \rangle -\langle \zeta_t, D_t h \rangle | + \sigma_{t+1}[h]\},
\end{multline*}
and, since
\[
\mathbb{E}|\langle \zeta_{t+1}, h \rangle -\mathbb{E}_t\langle \zeta_{t+1}, h \rangle |
\leq \mathbb{E}|\sigma_{t+1}^2[h]-\var_t\langle \zeta_{t+1}, h \rangle |^{1/2}
+ \sigma_{t+1}[h],
\]
it follows from (\ref{eq:SigmaApprox}) that
\[
|\mathbb{E}g(\langle \zeta_{t+1}, h \rangle) -\mathbb{E}g(\langle \tilde{\zeta}_{t+1}, h \rangle) |
\leq C \|g'\|_{\infty}(\sigma_{t+1}[h] + \|h\|_{\infty} \kappa_t n^{-1/4} \sqrt{1+\log n}).
\]
By this argument, and proceeding as in the proof of Proposition~\ref{prop:OneStep}, conditioning
instead on $\bv{X}_1,\dots,\bv{X}_t$,
we obtain Proposition~\ref{prop:GlobalBound}.
\begin{proposition}
\label{prop:GlobalBound}
There is a universal constant $C > 0$ such that, for any
$t \geq 0$, 
any function $g \in \mathcal{C}^1(\mathbb{R})$, and any
$h_1,\dots,h_{t+1} \in \mathbb{R}^n$,
\begin{multline*}
\left|\mathbb{E}g\left(\sum_{s=1}^t \langle \zeta_s, h_s \rangle + \langle \zeta_{t+1}, h_{t+1} \rangle\right)-
\mathbb{E}g\left(\sum_{s=1}^t \langle \zeta_s, h_s \rangle + \langle \tilde{\zeta}_{t+1}, h_{t+1}\rangle\right)\right| \\
\leq C \|g'\|_{\infty} 
\|h_{t+1}\|_{\infty} \kappa_t \cdot n^{-1/6} \sqrt{1 + \log n}. 
\end{multline*}
\end{proposition}
Through the Cram\`{e}r-Wold device, to prove Corollary~\ref{cor:CentralLim}, 
we prove instead the stronger result that,
for any collection of sequences $\{h_1^\n,\dots,h_t^\n\}_{n=1}^{\infty}$
with each $h_s^\n \in \mathbb{R}^n$ and $\|h_s^\n\|_{\infty} \leq H$
for some $H > 0$,
\begin{equation}
\label{eq:CentralLimProp}
\lim_{n\to\infty}
\left\lVert \law\left[\sum_{s=1}^{t}\langle \zeta_s^\n, h_{s}^\n\rangle\right]-
\law\left[\sum_{s=1}^{t}\langle \xi_s^\n, h_{s}^\n\rangle\right]\right\rVert_{1} = 0.
\end{equation}
Proceeding by induction on the case $t = 1$ which holds by 
Proposition~\ref{prop:GlobalBound}, assume (\ref{eq:CentralLimProp}) holds for some
$t \geq 1$. For each $n$, let $h_{t+1}^\n \in \mathbb{R}^n$ be such that
$\|h_{t+1}^\n\|_{\infty} \leq H$. 
Immediately, from Proposition~\ref{prop:GlobalBound},
\[
\lim_{n\to \infty}
\left\lVert \law\left[\sum_{s=1}^{t+1}\langle \zeta_s^\n, h_s^\n \rangle\right]-
\law\left[\sum_{s=1}^{t}\langle \zeta_s^\n, h_s^\n \rangle+\langle 
\tilde{\zeta}_{t}^\n, h_{t+1}^\n\rangle\right]\right\rVert_{1} = 0,
\]
and furthermore, from the induction hypothesis,
\begin{multline*}
\lim_{n\to\infty} \left\lVert \law\left[\sum_{s=1}^{t-1}\langle \zeta_s^\n, h_s^\n \rangle+
\langle \zeta_{t}^\n, h_t^\n + D_{t}^\n h_{t+1}^\n\rangle\right]\right.\\ - 
\left.\law\left[\sum_{s=1}^{t-1}\langle \xi_s^\n, h_s^\n \rangle+\langle \xi_{t}^\n, h_t^\n + D_{t}^\n h_{t+1}^\n\rangle\right]\right\rVert_{1} = 0.
\end{multline*}
But now the contraction identity (\ref{eq:Contraction}) tells us that
\[
\lim_{n\to\infty} \left\lVert \law\left[\sum_{s=1}^{t}\langle \zeta_s^\n, h_s^\n \rangle+\langle \tilde{\zeta}_{t}^\n, h_{s}^\n\rangle\right]-\law\left[\sum_{s=1}^{t+1}\langle \xi_s^\n, h_s^\n \rangle\right]\right\rVert _{1} = 0.
\]
The variance of $ \sum_{s=1}^{t+1}\langle \xi_s^\n, h_s^\n \rangle $ converges to the quantity given by Corollary~\ref{cor:CentralLim}, and hence the result holds for any $t \geq 1$.

\section{\label{sec:LongTerm}Long-term behaviour}
An advantage to working with approximations of stochastic
processes is that it is generally easier to study the
stationary and long-term behaviour of the approximation than it is the
original process. 
For example, we can identify a critical phase in
the stationary case, where $P_{i,t} = P_i$,
$i=1,\dots,n$, does not depend on $t$.
While Brouwer's Fixed Point Theorem guarantees 
that there is at least one fixed point,
most occupancy processes have a
trivial fixed point anyway (for example, corresponding to 
extinction).
Under sufficiently strict assumptions, it is possible to identify 
precisely globally stable equilibria of the deterministic system. We
extend partial relations to vectors by pointwise comparison, and say
that $\bv{x} \gneqq \bv{y}$ if $\bv{x} \geq \bv{y}$ and $\bv{x} \neq
\bv{y}$. Defining $J_0 = \lim_{\bv{x}\to 0^+} DP(\bv{x})$ where
$P = (P_1,\dots,P_n)$, and
letting $r(J_0)$ denote the spectral radius of the matrix
$J_0$, the results of \citep{Smith:1986aa} are interpreted within our
context as follows.

\begin{theorem}[Smith, 1986, Theorem 2.2]
\label{thm:SmithStability}
Suppose that, for every $\bv{x} > \bv{0}$ and $i,j=1,\dots,n$,
$\partial_j P_i(\bv{x}) > 0$, and that $DP(\bv{x}) \gneqq DP(\bv{y})$
for all $\bv{0} < \bv{x} < \bv{y}$. Assume also that $P_i(\bv{1}) \neq
1$ for some $i=1,\dots,n$. Then, the limit $\bv{p}_{\infty} :=
\lim_{t\to\infty} \bv{p}_t$ exists and is independent of $\bv{p}_0 \neq
\bv{0}$. Furthermore, $\bv{p}_{\infty} = \bv{0}$ if and only if $\bv{0}$
is a fixed point of $P$ and $r(J_0) \leq 1$.
\end{theorem}

If a globally stable equilibrium does indeed
exist, it is natural to consider the limit random variable
$\bv{Z}_{\infty}$. Fortunately, as an autoregressive process,
conditions for ergodicity of
$\bv{Z}_t$ are well-known \citep{Meyn:1993aa}. 
Let $V_t = \diag(p_{i,t}(1-p_{i,t}))_{i=1}^n$ and 
$\Sigma_t = \cov(\bv{Z}_t)$, and
observe that $\Sigma_t$ satisfies 
\[
\Sigma_{t+1} = DP(\bv{p}_t) \, \Sigma_t \, DP(\bv{p}_t)^\top + V_t.
\]
Let $V_{\infty} := \lim_{t\to\infty} V_t$ and $J_{\infty} =
DP(\bv{p}_\infty)$.
If $\Sigma_t$ converges to some matrix $Q$ as $t\to\infty$, then it must satisfy the discrete Lyapunov equation
\begin{equation}
Q = J_{\infty} Q J_{\infty}^\top + V_{\infty}. \label{eq:DiscreteLyapunov}
\end{equation}
If such a $Q$ exists, then, 
by vectorising~(\ref{eq:DiscreteLyapunov}), we see that it must also
satisfy
\begin{equation}
\label{eq:LyapunovVector}
\mbox{vec }Q = (I - J_{\infty} \otimes J_{\infty})^{-1} \mbox{vec }V_{\infty}.
\end{equation}
There are three immediate consequences of~(\ref{eq:LyapunovVector}): (i)
the solution to~(\ref{eq:DiscreteLyapunov}) is necessarily unique, (ii) a
solution exists if and only if every pair of eigenvalues
$\lambda_1$, $\lambda_2$ of $J_{\infty}$ satisfy $\lambda_1 \lambda_2 \neq 1$,
and, (iii) $Q = 0$ if and only if $\bv{p}_{\infty} \in \{\bv{0},\bv{1}\}$.
Altogether, we have

\begin{proposition}
\label{prop:NormalStability}
Assume that $\bv{p}_{\infty} = \lim_{t\to\infty} \bv{p}_t$ exists. Then,
$\Sigma_{\infty} := \lim_{t\to\infty} \Sigma_t$ exists if and only if
$r(J_{\infty}) < 1$, in which case $\Sigma_{\infty}$ is the solution
to~(\ref{eq:DiscreteLyapunov}). Furthermore, when $\Sigma_{\infty}$
exists, if $\bv{p}_{\infty} = \bv{0}$, then $\bv{Z}_t \convp \bv{0}$ as
$t\to \infty$; otherwise
\begin{equation}
\label{eq:LongTermNormal}
\bv{Z}_t \convd \mathcal{N}(\bv{p}_{\infty}, \Sigma_{\infty})\qquad\mbox{as }t\to\infty.
\end{equation}
\end{proposition}
Assuming the conditions of Proposition~\ref{prop:NormalStability}, while each
$\bv{Z}_t^\n$ is geometrically ergodic, to prove Corollary~\ref{cor:LongTerm}, 
it is necessary to show that the rate of convergence
of $\langle \xi_t^\n, h \rangle$ to $\langle \xi_{\infty}^\n, h \rangle$ is uniform in $n$. This is
accomplished in Lemma~\ref{lem:UnifConvT}.

\begin{lemma}
\label{lem:UnifConvT}
Suppose that $\sup_n \|J_{\infty}^\n\|_1 < 1$, $\Gamma^\n = \bigO(n^{-1})$,
and \linebreak $\sup_n \|\bv{p}_t^\n - \bv{p}_{\infty}^\n\|_{\infty} \to 0$ as
$t \to \infty$. For any $h \in \ell^{\infty}$,
there is a $\rho < 1$, and $C > 0$, independent of
$n, t$, such that, for every $t \geq 1$ and $n \in \mathbb{N}$,
\begin{align}
\label{eq:PiUnifConv}
|\langle \bar{\bv{p}}_t^\n - \bar{\bv{p}}_{\infty}^\n, h \rangle| & \leq C \|h\|_{\infty} \rho^{2 t} \\
\label{eq:XiUnifConv}
\|\law \langle \xi_t^\n, h \rangle-\law \langle \xi_{\infty}^\n, h \rangle\|_1
&\leq C \|h\|_{\infty} t \rho^t.
\end{align}
\end{lemma}
\begin{proof}
For each $n \in \mathbb{N}$, representing the process $\bv{Z}_t^\n$ by the
recursion
\begin{equation}
\label{eq:GeoRecur1}
\bv{Z}_{t+1}^\n - \bv{p}_{t+1}^\n = D_t (\bv{Z}_t^\n - \bv{p}_t^\n) + 
V_t^{1/2} \bv{z}_t,
\end{equation}
where $\bv{z}_t \iidsim \mathcal{N}(0,I)$  for each $t=0,1,\dots$, we
introduce a coupled process $\tilde{\bv{Z}}_t^\n$ satisfying
$\tilde{\bv{Z}}_0^\n = \bv{Z}_0^\n$ and
\begin{equation}
\label{eq:GeoRecur2}
\tilde{\bv{Z}}_{t+1}^\n - \bv{p}_{\infty}^\n
= J_{\infty}^\top (\tilde{\bv{Z}}_t^\n - \bv{p}_{\infty}^\n)
+ V_{\infty}^{1/2} \bv{z}_t.
\end{equation}
Additionally, by Proposition~\ref{prop:NormalStability}, $\bv{Z}_t^\n$ has
a limit as $t\to \infty$ for each $n \in \mathbb{N}$, which
we shall denote by $\bv{Z}_{\infty}^\n$ with $\cov(\bv{Z}_{\infty}^\n)
= \Sigma_{\infty}^\n$.
Denoting $\tilde{\xi}_t^\n = n^{-1/2}(\tilde{\bv{Z}}_t^\n - \bv{p}_t^\n)$
and $\xi_{\infty}^\n = n^{-1/2}(\bv{Z}_{\infty}^\n - \bv{p}_{\infty}^\n)$,
it
suffices to bound $\|\law \langle \xi_t^\n, h \rangle - \law \langle \tilde{\xi}_t^\n, h \rangle\|_1$
and $\|\law \langle \tilde{\xi}_t^\n, h \rangle - \law \langle \xi_{\infty}^\n, h \rangle\|_1$.
For the former, observe that
\begin{multline*}
\|J_{\infty}^\n-D_t^\n\|_1 
= \max_{j=1,\dots,n} \sum_{i=1}^n |\partial_j P_i^\n(\bv{p}_t^\n)
- \partial_j P_i^\n(\bv{p}_{\infty}^\n)| \\
\leq \Gamma^\n \|\bv{p}_t^\n - \bv{p}_{\infty}^\n\|_1,
\end{multline*}
and hence, for any $t \geq 0$,
\[
\|\bv{p}_{t+1}^\n - \bv{p}_{\infty}^\n\|_1
\leq \|J_{\infty}^\n\|_1 \|\bv{p}_t^\n - \bv{p}_{\infty}^\n\|_1
+ \Gamma^\n \|\bv{p}_t^\n - \bv{p}_{\infty}^\n\|_1^2.
\]
Letting $\rho^2 := \frac12 (1 + \sup_n \|J_{\infty}^\n\|_1) < 1$,
by assumption, there exists some $T$ such that, for every $t \geq T$
and $n \in \mathbb{N}$, $\|\bv{p}_t^\n-\bv{p}_{\infty}^\n\|_1
\leq (1 - \rho^2) / \Gamma^\n$. Therefore,
$\|\bv{p}_{t+1}^\n - \bv{p}_{\infty}^\n\|_1
\leq \rho^2 \|\bv{p}_t^\n - \bv{p}_{\infty}^\n\|_1$ for all $t \geq T$,
and hence, for every $n \in \mathbb{N}$,
\begin{equation}
\label{eq:GeoBoundP}
\|\bv{p}_t^\n - \bv{p}_{\infty}^\n\|_1 \leq C n \rho^{2 t},\qquad
\mbox{for all }t \geq 0,
\end{equation}
implying (\ref{eq:PiUnifConv}). 
Similarly, $\|D_t^\n\|_1 \leq \|J_{\infty}^\n\|_1 +
\|J_{\infty}^\n - D_t^\n\|_1 \leq \rho^2$ for all $t\geq T$,
and so, for all $0\leq s < t$,
\begin{equation}
\label{eq:DMatBounds}
\|D_{s,t}^\n\|_{\infty} \leq C \rho^{2(t-s)}, \qquad
\|\tilde{D}_{s,t}^\n\|_{\infty} \leq \rho^{2(t-s)}.
\end{equation}
Additionally, according to (\ref{eq:DiscreteLyapunov}),
\begin{equation}
\label{eq:LimitCovBound}
\|\Sigma_{\infty}^\n\|_1 \leq \frac{1}{1-\rho^2} \max_i p_{i,\infty}^\n(1-p_{i,\infty}^\n) \leq \frac{1}{4(1-\rho^2)}.
\end{equation}
From the recursion formulae~(\ref{eq:GeoRecur1}) and (\ref{eq:GeoRecur2})
and the identity $(a-b)^2 \leq |a^2 - b^2|$ for $a,b \geq 0$,
\[
\var[\langle \xi_t^\n - \tilde{\xi}_t^\n, h \rangle] 
\leq \frac1n \sum_{s=1}^t \|D_{s,t}^\n\|_{\infty}^2 \|h\|_{\infty}^2
\|\bv{p}_{\infty}^\n - \bv{p}_s^\n\|_1.
\]
Thus, from Jensen's inequality and inequalities~(\ref{eq:GeoBoundP}) and 
(\ref{eq:DMatBounds}), 
\[
\|\law \langle \xi_t^\n, h \rangle-\law \langle \tilde{\xi}_t^\n, h \rangle\|_1
\leq C \|h\|_{\infty} t \rho^t,
\]
as required. Finally, since $\langle\tilde{\xi}_t^\n,h\rangle$ and $\langle \xi_{\infty}^\n, h \rangle$
are normally distributed,
\begin{align*}
\|\law\langle \tilde{\xi}_t^\n, h \rangle-\law\langle \xi_{\infty}^\n, h \rangle\|_1^2
&\leq |\var\langle \tilde{\xi}_t^\n, h \rangle - \var\langle \xi_{\infty}^\n, h \rangle| \\
&\leq \|J_{\infty}^\n\|_1^{2t} \|h\|_{\infty}^2 \|\Sigma_{\infty}\|_1,
\end{align*}
and so (\ref{eq:LimitCovBound}) implies that
\[
\|\law\langle \tilde{\xi}_t^\n, h \rangle - \law\langle \xi_{\infty}^\n, h \rangle\|_1
\leq \|h\|_{\infty} \frac{\rho^t}{1 - \rho},
\]
and hence (\ref{eq:XiUnifConv}).
\end{proof}
The result of Lemma~\ref{lem:UnifConvT} 
implies that for any sequence of times~$\tau_n
\to \infty$
as $n \to \infty$, $\|\law\langle \xi_{\tau_n}^\n, h \rangle-\law\langle \xi_{\infty}^\n, h \rangle\|_1
\to 0$ as $n\to \infty$. Additionally, if $\tau_n \leq \frac{1}{17\alpha}
\epsilon \log n$ for $\epsilon > 0$, then Proposition~\ref{prop:GlobalBound} implies that
\[
\|\law\langle \zeta_{\tau_n}^\n, h \rangle - \law\langle \xi_{\tau_n}^\n, h \rangle\|_1
= \bigO(1) \|h\|_{\infty} n^{-1/6+\epsilon} (1+\log n)^{3/2} \to 0 \mbox{ as } n \to \infty.
\]
Together, these two facts imply Corollary~\ref{cor:LongTerm}.

\section{Applications}
\label{sec:Examples}
We now apply our results to a variety of existing models. 
\begin{example}[Spreading Processes]
\label{ex:Spreading}
To demonstrate the utility of Theorem~\ref{thm:LqrErrorBounds} 
in identifying the critical phase of occupancy
processes, we consider the time-homogeneous contact-based epidemic
\emph{spreading processes} introduced by Wang et
al.~\citep{Wang:2003aa}, and generalised in~\citep{Gomez:2010aa}. A survey
of more recent applications, and extensions to multi-layer networks can
be found in \citep[Section 5.2]{Boccaletti:2014aa}. This class of
processes encompasses those amenable to heterogeneous mean field
approaches, and allows for both weighted and unweighted networks.
Additionally, there are a number of recent social network
\citep{Wei:2013aa} and computer science \citep{Mirchev:2010aa} models
which fall within this framework. They may be summarised as follows.
First it is assumed that the probability of node~$i$ (of a total of~$n$)
being infected by an infected node~$j$ in one time step is $r_{ij}$,
with the convention that $r_{ii} = 0$.
The collection $R = (r_{ij})$ is called the \emph{reaction matrix}.
For the
case of a single-layer network with a weighted adjacency matrix $W =
(w_{ij})$, by defining $\lambda_i$ as the number of contacts from node
$i$ per unit time, $R$ may be defined by
\[
r_{ij} \propto 1 - \left(1 - \frac{w_{ij}}{\sum_{j=1}^n w_{ij}}\right)^{\lambda_i},
\]
where the constant of proportionality is assumed to be independent 
of $i$ and $j$.
For interconnected and multiplex networks, reaction matrices become
significantly more complex in form \citep{Boccaletti:2014aa}. Assuming
that contacts are all independent, the colonisation function $C_i$ is
given by
\[
C_i(\bv{x}) = 1 - \prod_{\substack{j=1\\j\neq i}}^n (1 - r_{ij} x_j).
\]
If it is assumed that a node recovers with probability $\mu$ in one time
step, then we may take $S_i(\bv{x}) = 1- \mu$. However, as in
\citep{Gomez:2010aa}, we may also consider the possibility of reinfection
before the next census, giving a survival function of the form
\[
S_i(\bv{x}) = 1 - \mu [1 - C_i(\bv{x})] 
= 1 - \mu \prod_{\substack{j=1\\j\neq i}}^n (1 - r_{ij} x_j).
\]
It is a straightforward exercise to show $\alpha = \|R\|_1$ and $\psi = 
n^{-1/2} \|R\|_F$. 

Let $\bar{X}_t = n^{-1} \sum_{i=1}^n X_{i,t}$ denote the total proportion of infectives
at time~$t$, with $\bar{p}_t = n^{-1} \sum_{i=1}^n p_{i,t}$ 
its deterministic approximation constructed
according to (\ref{eq:Detalternative}). 
Theorem~\ref{thm:LqrErrorBounds}
gives,
\begin{equation}
\label{eq:SpreadingApprox}
\mathbb{E}|\bar{X}_t - \bar{p}_t| \leq C n^{-1/2} t (1 + \|R\|_F)
e^{4 t \|R\|_1}.
\end{equation}
Now, $\Gamma$ is
the largest element of the matrix $(R+I)^\top (R+I) - I$, while $\delta = 0$.
Assuming $r_{ij} \leq \bar{r}n^{-1}$ for each $i,j=1,\dots,n$ for some
$\bar{r} > 0$, we have
$n\Gamma \leq (1 + \bar{r})^2$ and $\|R\|_1,\,\|R\|_F \leq \bar{r}$. Denoting
$v_t = \sum_{s=1}^t \sigma_s^2[D_{s,t}\bv{1}] \leq \frac14 t e^{t\|R\|_1}$,
Theorem~\ref{thm:WasserBound} implies
\begin{equation}
\label{eq:SpreadingNormal}
\left\lVert\law (\bar{X}_t) -\mathcal{N}(\bar{p}_t,n^{-1} v_t)\right\rVert_1 \leq 
\frac{C(1+\log n)^{1/2}}{n}\,te^{19\bar{r}t}(1+\bar{r})^{4}\sum_{s=1}^{t}\frac{e^{-3\bar{r}s}}{\sigma_{s}^{2}[D_{s,t}\bv{1}]}.
\end{equation}
Alternatively, for the same occupancy process, the global rule may be extended
to $[0,1]^n$ in the form (\ref{eq:HanskiForm}) with $f_i(x) = 1-\mu e^{-x}$,
$g_i(x) = 1-e^{-x}$ and $s_{ij}=c_{ij}=n|\log(1-r_{ij})|$. In this case,
(\ref{eq:SpreadingNormal}) may be improved to $\bigO(n^{-1})$ by 
Remark~\ref{rem:MeanField}, although this comes at the cost of larger $\alpha,\psi,\Gamma,\delta$. 

G{\'o}mez et al.~\citep{Gomez:2010aa} showed that if the spectral radius
$r(R)$ is strictly less than $\mu$, 
the disease, as represented through the deterministic
system~(\ref{eq:Detalternative}), cannot become endemic. This is seen by
interpreting an epidemic as a non-zero fixed point of the deterministic
recurrence. But, since $J_0 = (1-\mu)I + R$, 
Theorem~\ref{thm:SmithStability} and~(\ref{eq:SpreadingApprox})
imply that, when
$n$ is large and $\|R\|_1,\,\|R\|_F$ are not, 
the process $\bv{X}_t$ quickly reaches the 
fixed point $\bv{X}_t = \bv{0}$ (corresponding to
the infection dying out) if $r(R) \leq \mu$, and may persist otherwise.
In the latter case, using Perron-Frobenius theory, 
it may be shown that $r(J_{\infty}) < 1$
holds assuming $r_{ij} > 0$ for all $i,j=1,\dots,n$, $i \neq j$. 
If the deterministic equilibrium satisfies
\[
[1-(1-\mu)p_{j,\infty}]\sum_{\substack{i=1\\i\neq j}}^n \frac{r_{ij}(1-p_{i,\infty})}{1-r_{ij}p_{j,\infty}} < \mu,
\]
then $\|J_{\infty}\|_1 < 1$, and so 
convergence of the recentered process to a normal equilibrium
follows from Corollary \ref{cor:LongTerm}.
\end{example}

\begin{example}[Hanski's Incidence Function Model]
\label{ex:Hanski}
Arguably the first, and perhaps the most widely used and studied
stochastic patch occupancy model, is the \emph{Incidence Function
Model} introduced by Hanski~\citep{Hanski:1994}. Recent work 
by \allowbreak McVinish \& Pollett~\citep{McVinish:2014a}
has considered the model within the occupancy process framework.

We present a time-inhomogeneous extension of the general formulation of
Hanski's model, which may be realized in our framework in the following
way. Let $\Omega \subset \mathbb{R}^d$ be a compact set, and associated
with each patch~$i$ is a location $z_i$ in $\Omega$. The survival function
$S_{i,t}$ for each patch $i$ is chosen to be independent of all other
patches, so that $S_{i,t}(\bv{x}) = s_{i,t}$ for each $i=1,\dots,n$ and
$t=0,1,\dots$. For the colonisation function $C_{i,t}$, let $c
:\,[0,\infty) \to [0,1]$ be a $\mathcal{C}^2$ function, and, for each
$t=0,1,\dots$ and $i=1,\dots,n$, let $A_{i,t}$ denote the colonisation
weight of patch $i$ at time $t$, corresponding either to patch size or
approximate population size. Then, for some $D:\,\Omega^2\to\mathbb{R}$, 
we let
\begin{equation}
\label{eq:HanskiCol}
C_{i,t}(\bv{x}) = c\left(\sum_{j\neq i}^n A_{j,t} D(z_i,z_j) x_j\right).
\end{equation}
The inner sum is often referred to as the \emph{connectivity measure} for
patch $i$.

To investigate the asymptotic behaviour of this system as the number of
patches grows, we consider a sequence of patch locations
$\{z_i\}_{i=1}^{\infty}$ that is equidistributed in $\Omega$ with
respect to a distribution $m$, so that, for every $h\in\mathcal{C}(\Omega)$,
$n^{-1} \sum_{i=1}^n h(z_i) \to \int h \, \mathd m$, as $n\to\infty$.

Suppose that, for each $n\in\mathbb{N}$, patches are placed at locations
$z_1,\dots,z_n$ with patch weights $A_{i,t} = n^{-1} a_t(z_i)$, where
$a_t \in \mathcal{C}(\Omega)$ describes weight density at time $t$.
Similarly, assume that the survival probabilities $s_{i,t} = s_t(z_i)$
where $s_t \in \mathcal{C}(\Omega)$ describes the probability of survival
at time $t$ according to locations in $\Omega$.

We now represent our functionals $\mu_t^\n$ and $\pi_t^\n$ as measures,
defined for $h\in\mathcal{C}(\Omega)$ by $\int h \mathd \mu_t^\n = n^{-1}
\sum_{i=1}^n h(z_i) X_{i,t}^\n$, and similarly for $\pi_t^\n$. Equip the
space $\mathcal{M}(\Omega)$ of finite measures acting on the Borel
$\sigma$-algebra of $\Omega$ with the vague topology, so that for
$\{\mu_n\}_{n=1}^{\infty} \subset \mathcal{M}(\Omega)$ and
$\mu\in\mathcal{M}(\Omega)$, and $\mu_n \convw \mu$ if and only if $\int h
\mathd \mu_n \to \int h \mathd \mu$ for all $h\in\mathcal{C}(\Omega)$
\citep[Theorem 16.16]{Kallenberg:2002}. 
If $D$ is an equicontinuous function (that is,
for every $\epsilon > 0$, there is a $\delta > 0$ such that if
$z,z_1,z_2 \in \Omega$ with $\|z_1-z_2\| < \delta$, then
$|D(z_1,z)-D(z_2,z)|<\epsilon$), the proof of \citep[Theorem
1]{McVinish:2014a} implies the following lemma.

\begin{lemma}
\label{lem:HanskiLim}
If $\pi_0^\n \convw \pi_0$ as $n\to\infty$ for some $\pi_0 \in
\mathcal{M}(\Omega)$ that is absolutely continuous with respect to~$m$,
then $\pi_t^\n \convw \pi_t$ for each $t\in\mathbb{N}$,
where $\pi_t \in \mathcal{M}(\Omega)$ is defined recursively in terms of
its Radon-Nikodym derivative by
\begin{equation}
\label{eq:HanskiEvolve}
\frac{\partial\pi_{t+1}}{\partial m}(z) = s_t(z)\cdot\frac{\partial\pi_t}{\partial m}(z) + c[C_t(z)] \cdot \left[1 - \frac{\partial \pi_t}{\partial m}(z)\right],
\end{equation}
where $C_t(z) = \int a_t(\tilde{z}) D(z,\tilde{z}) \pi_t(\mathd \tilde{z})$ is the limiting connectivity measure at time $t$.
\end{lemma}

Since $\{h(z_i)\}_{i=1}^{\infty} \in \ell^{\infty}$ for every $h \in \mathcal{C}(\Omega)$ and $\|D\|_{\infty} < \infty$, a quick application of Corollary~\ref{cor:LawLargeNums} with \citep[Theorem 2.2]{Berti:2006aa} under the assumptions of Lemma \ref{lem:HanskiLim} yields the following result, which improves \citep[Theorem 1]{McVinish:2014a}.

\begin{proposition}
\label{prop:HanskiLLN}
For every $t=1,2,\dots$, $\mu_t^\n \convw \pi_t$ almost surely.
\end{proposition}

Proceeding further, it may be shown that the normalised fluctuations in
$\mu_t^\n$ converge to a Gaussian random field, which, to our knowledge,
is an entirely new result. For any $h\in\mathcal{C}(\Omega)$, define
$\sigma_{n,t}^2:\,\mathcal{C}(\Omega)\to[0,\infty)$ by
$\sigma_{n,t}^2[h]:=n^{-1}\sum_{i=1}^n h(z_i)^2 p_{i,t}^\n (1 -
p_{i,t}^\n)$. It is straightforward to check that, for every
$t=0,1,\dots$ and $h \in \mathcal{C}(\Omega)$, $\sigma_{n,t}[h]^2 \to
\int h^2 \mathd \sigma_t$ as $n\to\infty$, where the measures $\sigma_t$ on
$\Omega$ are absolutely continuous with respect to $m$ and are defined
recursively according to $\sigma_0 \equiv 0$ and
\begin{align*}
\frac{\partial\sigma_{t+1}}{\partial m}(z)&=s_t(z)[1-s_t(z)]\frac{\partial\pi_{t}}{\partial m}(z)+c\left[C_{t}\left(z\right)\right]( 1-c[C_{t}(z)]) \left(1-\frac{\partial\pi_{t}}{\partial m}(z)\right) \\
 & \qquad + (s_t(z)-c[C_{t}(z)])^{2}\cdot\frac{\partial\sigma_{t}}{\partial m}(z).
\end{align*}
Additionally, for any $h \in \mathcal{C}(\Omega)$ and $j=1,2,\dots$,
we have
$$
\sum_{i=1}^n h(z_i) \partial_j P_{i,t}^\n(\bv{p}_t^\n) \to
(\mathcal{J}_t h)(z_j)
$$
as $n\to\infty$, where
$\mathcal{J}_t:\,\mathcal{C}(\Omega)\to\mathcal{C}(\Omega)$ is defined
by
\begin{align*}
\mathcal{J}_t h(z) &= (s_t(z)-c[C_t(z)])h(z) \\ & + a_t(z) \int D(z,\tilde{z})h(\tilde{z})c'[C_t(\tilde{z})](m-\pi_t)(\mathd\tilde{z}).
\end{align*}
Thus, Corollary~\ref{cor:CentralLim} implies the following central limit result.
\begin{proposition}
\label{prop:HanskiCLT}
For each $t=1,2,\dots$, and $h \in \mathcal{C}(\Omega)$, 
\begin{equation}
\label{eq:HanskiCLT}
\sqrt{n}\left(\int h \mathd\mu_t^\n - \int h
\mathd\pi_t^\n\right) 
\convd \sum_{s=0}^t \int \mathcal{Q}_{t-s} h \, \mathd\tilde{\xi}_s,
\end{equation}
where $\mathcal{Q}_t :\, \mathcal{C}(\Omega) \to \mathcal{C}(\Omega)$ is
defined by
$\mathcal{Q}_t := \mathcal{J}_t \circ \mathcal{J}_{t-1} \circ \cdots
\circ \mathcal{J}_1$,
and each $\tilde{\xi}_s$ is an independent Gaussian white noise 
on $\Omega$ with intensity measure $\sigma_s$.
\end{proposition}
For each $t=1,2,\dots$ and $n=1,2,\dots$, let $\zeta_t^\n$ and $\xi_t$
denote the random signed measures defined through their integrals with
respect to functions $h \in \mathcal{C}(\Omega)$ by the left and right-hand
sides of (\ref{eq:HanskiCLT}) respectively. 
Ideally, the convergence (\ref{eq:HanskiCLT}) could be represented in
a concise fashion as in Proposition~\ref{prop:HanskiLLN}. Unfortunately,
the space of signed measures endowed with the weak topology is not
metrisable \citep{Bansaye:2011aa}, prohibiting conventional convergence theorems
in this setting. Instead, it is common to embed a signed random measure
into the dual of a Sobolev space. For $1 \leq r < \infty$, let $W^r(\Omega)$
denote the Sobolev space of order $r$ on $\Omega$, defined as the closure
of $L^2(\Omega) \cap \mathcal{C}^{\infty}(\Omega)$ under the norm
\[
\|u\|_{W^r(\Omega)}^2 := \sum_{|\alpha| \leq r} \int_\Omega |\partial^\alpha u(x)|^2 \mathd x,
\]
where the sum is taken over all multi-indices $\alpha=(\alpha_1,\dots,\alpha_d)$
with $|\alpha|=\sum_i \alpha_i \leq r$, and the derivatives are understood
in the weak sense. Assuming $\partial\Omega$ is locally Lipschitz, by the Sobolev embedding theorem \citep[Theorem 4.12]{Adams:2003aa}, $W^r(\Omega) \subset
\mathcal{C}(\Omega)$ for all $r > d/2$, and there is a constant $C > 0$
depending only on $\Omega$ and $r$ such that
\begin{equation}
\label{eq:SobolevEmbed}
\|u\|_{\infty} \leq C \|u\|_{W^r(\Omega)},\qquad u \in W^r(\Omega),
\end{equation}
with some mild abuse of notation. 
Let $W^{-r}(\Omega)$ denote the dual space of $W^r(\Omega)$,
observing that any signed random measure $\zeta$ may be identified as an
element in $W^{-r}(\Omega)$ by $(h,\zeta):= \int_\Omega h \mathd \zeta$ for
$h \in W^r(\Omega)$. 
Of course, if $\zeta_n \convd \zeta$ where $\zeta,\zeta_1,\zeta_2,\dots \in
W^{-r}(\Omega)$, then $(h,\zeta_n) \convd (h,\zeta)$ for any $h \in W^r(\Omega)$.
By virtue of Theorem~\ref{thm:LqrErrorBounds} and~(\ref{eq:SobolevEmbed}), for any $h \in W^r(\Omega)$ and each $t=1,2,\dots$,
\begin{equation}
\label{eq:HanskiCompact}
\sup_n \mathbb{E}|(h,\zeta_t^\n)| \leq C_t \|h\|_{W^r(\Omega)}\quad \mbox{and}\quad
\sup_n \mathbb{E}\|\zeta_t^\n\|_{W^{-r}(\Omega)} < \infty.
\end{equation}
Since the embedding $W^{r+1}(\Omega) \hookrightarrow W^r(\Omega)$ is Hilbert-Schmidt, any closed ball in $W^{-r}(\Omega)$ is compact in $W^{-r-1}(\Omega)$. Therefore, by Markov's inequality, (\ref{eq:HanskiCompact}) implies that $\{\zeta_t^\n\}_{n=1}^{\infty}$ is tight in
$W^{-r-1}(\Omega)$, whence Proposition~\ref{prop:HanskiCLT} implies a concise
Corollary~\ref{cor:HanskiCLT2}.
\begin{corollary}
\label{cor:HanskiCLT2}
For each $t=1,2,\dots,$ the signed random measures $\zeta_t^\n \convd
\xi_t$ in $W^{-r}(\Omega)$ for any $r > d/2 + 1$. 
\end{corollary}

\end{example}
\begin{example}[Dynamic Random Graphs]
\label{ex:RandomGraph}
The prototypical representation of a complex network is that of a
\emph{random graph} of large size. In a stochastic setting, one can
construct very general processes on a space of graphs to model the
evolution of large networks \citep{Durrett:2007}, but such processes are often difficult to
study. It is convenient then that many dynamic random graphs can be
formulated as occupancy processes, where now the nodes become the
vertices of a line graph, describing the presence of an edge.
The natural setting for analysing large dense random graphs is by
way of \emph{graphons}, defined as symmetric Borel measurable functions
$W:\,[0,1]^2 \to [0,1]$. Any graph $G=(V(G),E(G))$ with $V(G) = \{1,2,\dots,m\}$
may be embedded in a graphon $W_G$ by subdividing $[0,1]$ into
intervals $I_1,\dots,I_m$ and taking $W_G(x,y) = \ind_{ij \in E}$ for
all $(x,y) \in I_i \times I_j$, $i,j=1,\dots,m$. In the sequel, 
the standard graphon of $G$ is to be regarded as the case where each
subinterval has equal length.
Many properties of
a graph may be formulated in terms of its graphon, including the degree
function, defined, for vertex $i$ with $x_i \in I_i$, by
\[
d_{W_G}(x_i) = \int_0^1 W_G(x_i,y) \mathd y = \frac{\deg i}n.
\]
But perhaps the most important object in the study of large dense random
graphs is the \emph{homomorphism density}: if $F$ is a simple graph
and $W$ a graphon, we define the homomorphism density $t(F,W)$ of $F$
into $W$ by
\[
t(F,W) = \int_{[0,1]^{V(F)}} \prod_{ij\in E(F)} W(x_i,x_j) \prod_{i \in V(F)}
\mathd x_i.
\]
Naturally, graphons are strict generalisations
of graphs, which becomes important for developing limit theorems.
In this connection, the homomorphism density provides a good starting point;
if a sequence of graphons $W^\n$ converges to $W$ in a reasonable sense,
we might expect that $\lim_{n\to\infty} t(F,W^\n) = t(F,W)$ for any
simple graph $F$. 
It turns out that to show convergence of the underlying graphs in homomorphism density,
it is sufficient to show convergence under the cut metric
\citep[Lemma 10.23]{Lovasz:2012aa}, induced by the norm
$\|\cdot\|_{\square}$, defined for \emph{kernels} $W : [0,1]^2 \to \mathbb{R}$ by
\[
\|W\|_{\square} = \sup_{U,V\subseteq[0,1]}\left|\int_{U\times V} W(x,y)
\, \mathd x \mathd y\right|.
\]
For more details on graphons, refer to the comprehensive monograph of 
Lov\`{a}sz~\citep{Lovasz:2012aa}.

To illustrate, consider the following sequence of
Markov preferential attachment models. For each $n=1,2,\dots$, let
$G^\n$ be a simple graph with $n$ edges and $v_n$ vertices, and suppose
that $G^\n \subset G^{(n+1)}$. Now, let
$G_t^\n$, $t=0,1,\dots$, be a Markov chain on the space of subgraphs
of $G^\n$ whose edges evolve independently in such a way that each
edge $ij \in E(G^\n)$ is deleted with probability $q_t$ and is
added with probability
\[
f\left(\frac{1}{2v_n}\deg i + \frac{1}{2v_n} \deg j\right),
\]
for some function $f:\,[0,1]\to[0,1] \in \mathcal{C}^2$. Since
$\alpha^\n \leq \|f'\|_{\infty}$ and $\psi^\n \leq v_n^{-1}\|f'\|_{\infty}
+ v_n^{-2}\|f'\|_{\infty}$, it can be verified that this sequence of
occupancy processes satisfies the conditions of 
Corollaries~\ref{cor:LawLargeNums} and~\ref{cor:CentralLim}. 
We begin by showing almost sure convergence of the underlying graphs under the
cut metric, which is easily achieved with the help of Corollary~\ref{cor:LawLargeNums}.

\begin{proposition}
\label{prop:GraphonConv}
Let $W^\n$ and $W_t^\n$ denote the graphons of $G^\n$ and $G_t^\n$ (respectively)
and suppose that $\|W^\n-W\|_{\square}\to 0$ and
$\|W_0^\n - W_0\|_{\square}\to~0$ as $n\to\infty$ for some graphons $W$ and $W_0$. 
Define the sequence of graphons $W_t$ by the recursion
\begin{equation}
\label{eq:GraphonRecur}
\small{W_{t+1}(x,y) = q_t W_t(x,y) + W(x,y)[1-W_t(x,y)] f\left(\frac{d_{W_t}(x)
+d_{W_t}(y)}{2}\right).}
\end{equation}
Then, for each $t=1,2,\dots$, $\|W_t^\n - W_t\|_{\square} \convas 0$. 
\end{proposition}
\begin{proof}
For each $t=0,1,\dots$ define the graphon $\tilde{W}_t^\n$ by 
$\tilde{W}_0^\n = W_0^\n$ and satisfying the same recursion 
relation~(\ref{eq:GraphonRecur}). It is a straightforward
exercise 
to show that $\|\tilde{W}_t^\n - W_t\|_{\square} \to 0$ as $n\to\infty$.
Observe that for $\mathcal{A}_n = \{uv^\top\,:\, u,v \in \{0,1\}^{v_n}\}$ the
set of $v_n\times v_n$ binary rank-one matrices,
\[
\|W_t^\n - \tilde{W}_t^\n\|_{\square} = \max_{A \in \mathcal{A}_n}
\left|\frac1{v_n^2} \sum_{i,j=1}^{v_n} A_{ij} [W_t^\n(x_i,x_j) - \tilde{W}_t^\n(x_i,x_j)]\right|.
\]
Since $\log_2|\mathcal{A}_n|\leq 2 v_n$, from (\ref{eq:RadBounds}), $\Rad(\mathcal{A}_n) = \bigO(v_n^{-1/2})$ as $n \to \infty$, and so the result follows from Corollary~\ref{cor:LawLargeNums}.
\end{proof}

Proceeding further, a combination of Proposition~\ref{prop:LinEstimate} and Corollary~\ref{cor:CentralLim} show that arbitrary homomorphism
densities $t(F,W_t^\n)$ of $G_t^\n$ are approximately normally distributed about
$t(F,W_t)$. To illustrate, consider the density of triangles in $G_t^\n$ given
by $t(\triangle,W_t^\n)$ where $\triangle$ is the complete graph on three vertices.
\begin{proposition}
\label{prop:HomoCLT}
Assume that $|t(\triangle,W_0^\n)-t(\triangle,W_0)| = o(v_n^{-1}n^{1/2})$ as $n \to \infty$. 
Then, for each $t \geq 1$, as $n \to \infty$,
\[
v_n n^{-1/2} [t(\triangle,W_t^\n) - t(\triangle,W_t)] \convd \mathcal{N}(0,\mathcal{V}_t[\Lambda_t]),
\]
where $\Lambda_t:[0,1]^2\to[0,3]$ is a kernel defined by
\[
\Lambda_t(x,y) = 3 \int_0^1 W_t(x,z) W_t(z,y) \mathd z.
\]
Here $\mathcal{V}_t$ is a functional acting on kernels $U:[0,1]^2\to\mathbb{R}$ satisfying $\mathcal{V}_0 \equiv 0$ and, for $t=0,1,\dots$, the recursion relation $\mathcal{V}_{t+1} U=\sigma_{t+1}^2 U+\mathcal{V}_t \circ \mathcal{J}_t U$, with
\begin{multline*}
\mathcal{J}_t U(x,y) := W(x,y)\left\lbrace U(x,y)(q_t - f[\tfrac12 (d_U(x)
+d_U(y))])\vphantom{\int_0^1}\right. \\
+ \int_0^1 W(x,z)[1-U(x,z)]f'[\tfrac12(d_U(x)+d_U(z))] \\
\left. \vphantom{\int_0^1}+ W(y,z)[1-U(y,z)]f'[\tfrac12(d_U(y)+d_U(z))]\, \mathd z\right\rbrace,
\end{multline*}
while $\sigma_t^2$ is a functional acting on kernels satisfying $\sigma_0^2 \equiv 0$ and for each $t=0,1,\dots$,
\begin{multline*}
\sigma_{t+1}^2[U(x,y)] = q_t(1-q_t)\int_{[0,1]^2} U(x,y) W_t(x,y) \mathd x \mathd y \\
+ \int_{[0,1]^2} U(x,y) w_t(x,y)[1-w_t(x,y)][1-W_t(x,y)] \mathd x \mathd y \\
+ \sigma_t^2[(q_t - w_t(x,y))U(x,y)],
\end{multline*}
with $w_t(x,y) = f[\tfrac12(d_{W_t}(x)+d_{W_t}(y))]$. 
\end{proposition}
\begin{proof}
Defining the triangle density acting on adjacency matrices $A \in
\{0,1\}^{v_n\times v_n}$, and computing derivatives, we find that
\[
t(\triangle,A) = \frac{1}{v_n^3}\sum_{i,j,k=1}^{v_n} A_{ij}A_{jk}A_{ik},\quad
\partial_{ij}t(\triangle,A) = \frac{3}{v_n^3} \sum_{k=1}^{v_n} A_{ik}A_{jk}.
\]
Constructing $\tilde{W}_t^\n$ once again as in the proof of Proposition~\ref{prop:GraphonConv}, Proposition~\ref{prop:LinEstimate} implies that
\[
\mathbb{E}|t(\triangle,W_t^\n) - t(\triangle,\tilde{W}_t^\n)
-(\Lambda_t^\n,\zeta_t^\n)| = \bigO(n v_n^{-3})\qquad\mbox{as }n\to\infty,
\]
where $\Lambda_{ij}^\n = 3n^{1/2}v_n^{-3}\sum_{k=1}^{v_n} W_t^\n(x_i,x_k)W_t^\n(x_j,x_k)$. 
By induction on the assumed case $t = 0$, it may be shown that
$|t(\triangle,\tilde{W}_t^\n)-t(\triangle,W_t)|=o(v_n^{-1} n^{1/2})$ for each
$t=1,2,\dots$. 
Therefore, it remains only to show convergence in law of $v_n^2 n^{-1/2} (\Lambda_t^\n,\zeta_t^\n)$. Taking $n\to\infty$ and relabelling as necessary, $v_n^2 n^{-1/2} \Lambda_{ij,t}^\n \to \Lambda_t(x_i,x_j)$.
The rest is implied by (\ref{eq:CentralLimProp}), following similar computations to those in Example~\ref{ex:Hanski}. 
\end{proof}
To adapt Proposition~\ref{prop:HomoCLT} to homomorphism densities from
a different simple graph $F$, one need only modify the kernel $\Lambda_t$ ---all 
other objects remain intact. 
Allowing one final remark, it is also quite possible to consider another occupancy process running on the nodes of a dynamic random graph model as one conglomerate occupancy
process. While notation becomes rather unwieldy at this level of complexity, provided that presence/absence of edges and the states of the vertices are
not too intimately connected as to violate the assumptions of Corollary~\ref{cor:CentralLim}, many of the ideas contained in Examples~\ref{ex:Spreading} and~\ref{ex:RandomGraph} should extend to the more
general setting. 
\end{example}

\section*{Acknowledgements}
L. Hodgkinson thanks Andrew Barbour for numerous valuable conversations
concerning Stein's method, with particular appreciation for 
his insight into the one-step mechanism which makes
Theorem \ref{thm:WasserBound} possible. He is also grateful to Aihua Xia and Nathan Ross for their friendly advice on approaching problems using Stein's method.

\appendix

\section{The method of bounded differences}
\label{sec:MomentIneq}
The classical \emph{method of bounded differences} provides the simplest and most
versatile approach for developing moment estimates and concentration inequalities
involving bounded random variables. We recall a few results
from the theory which will be greatly useful to us. 
For any function $f:\{0,1\}^n \to \mathbb{R}$, let $\Delta_i f(\bv{x}) = 
|f(\bv{x})-f(\bv{x}^i)|$, where $\bv{x}^i$ denotes $\bv{x}$ with the $i$-th 
component replaced by its inverse $x_i^i = 1 - x_i$. 
Let $\Psi(x) = e^{x^2} - 1$ 
and $\|\cdot\|_{\Psi}$ be the corresponding Orlicz norm, defined for a random 
variable $X$ by
\[
\|X\|_{\Psi} = \inf\{t > 0\,:\,\mathbb{E}\Psi(X/t) \leq 1\}.
\]
The Orlicz norm is particularly useful for controlling maxima: if $X_1,\dots,X_n$ are random variables, not
necessarily independent, then \citep{Pollard:1989aa}
\begin{equation}
\label{eq:MaxOrlicz}
\mathbb{E}\max_i X_i^2 \leq \log 2 n \cdot \max_i \|X_i\|_{\Psi}^2.
\end{equation}
\begin{theorem}[\textsc{Method of Bounded Differences}]
\label{thm:BernFuncMoments}
Let $W_1,\dots,W_n$ be independent $\{0,1\}$-valued random variables and
$\bv{W} = (W_1,\dots,W_n)$. For any 
$f:\,\{0,1\}^n \to \mathbb{R}$, let $\|\Delta f\|_2^2 = 
\sum_{i=1}^n \|\Delta_i f\|_{\infty}^2$. Then, for any $\epsilon > 0$, 
and $q \geq 1$,
\begin{align}
\label{eq:McDiarmid}
\mathbb{P}(f(\bv{W}) > \epsilon + \mathbb{E}f(\bv{W}))
&\leq \exp\left(-\frac{2 \epsilon^2}{\|\Delta f\|_2^2}\right) \\
\label{eq:BernMoment}
\|f(\bv{W})-\mathbb{E}f(\bv{W})\|_{q\,\,} &\leq \sqrt{\tfrac{\pi q}{2}} \|\Delta f\|_2 \\
\label{eq:BernOrlicz}
\|f(\bv{W})-\mathbb{E}f(\bv{W})\|_{\Psi} & \leq \sqrt{\tfrac{3}{2}} 
\|\Delta f\|_2.
\end{align}
\end{theorem}
The `one-sided' inequality (\ref{eq:McDiarmid}) is famously due to
McDiarmid \citep[Theorem 3.1]{McDiarmid:98}. 
Adding (\ref{eq:McDiarmid}) to itself
applied to $-f$ yields the `two-sided' inequality for tail estimates
of $|f(\bv{W})-\mathbb{E}f(\bv{W})|$,
from which (\ref{eq:BernMoment}) and
(\ref{eq:BernOrlicz}) follow by calculation.

To compare $\mathbb{E}f(\bv{W})$ with $f(\mathbb{E}\bv{W})$, we can make
use of a simple telescoping trick dating back to Lindeberg's original
analytic proof of the central limit theorem. For more recent applications
of this trick, we refer the reader to the paper of Chatterjee
\citep{Chatterjee:2006aa}.
Let $f \in \mathcal{C}^2([0,1]^n)$ and $W_1,\dots,W_n$ be independent
random variables with $p_i = \mathbb{E}W_i$ for each $i=1,\dots,n$.
Now, for each $i=0,\dots,n$, let $\tilde{\bv{W}}_i =
(W_1,\dots,W_i,p_{i+1},\dots,p_n)$, so that, from Taylor's Theorem,
\begin{align*}
|\mathbb{E}f(\bv{W}) - f(\bv{p})| &\leq \sum_{i=1}^n |\mathbb{E}f(\tilde{\bv{W}}_i)
-\mathbb{E}f(\tilde{\bv{W}}_{i-1})| \\
&\leq \sum_{i=1}^n |\mathbb{E}[\partial_i f(\tilde{\bv{W}}_{i-1})\cdot(W_{i,t}-p_{i,t})]| + \frac12 \sum_{i=1}^n \|\partial_i^2 f\|_{\infty}.
\end{align*}
But, since $\tilde{W}_{i-1}$ is independent of $W_i$, the first term is identically
0, and
\begin{equation}
\label{eq:Lindeberg}
|\mathbb{E}f(\bv{W}) - f(\bv{p})| \leq \frac12 
\sum_{i=1}^n \|\partial_i^2 f\|_{\infty}.
\end{equation}
Theorem~\ref{thm:BernFuncMoments} together with (\ref{eq:Lindeberg}) provides an
effective measure on the deviation of $f(\bv{W})$ from $f(\bv{p})$ for any
arbitrary $f \in \mathcal{C}^2([0,1]^n)$. We shall also find it useful to
perform a linear approximation to $f$ as an intermediary to computing $f(\bv{W})$,
and bound the error incurred in doing so.
This requires a tighter estimate than is offered in the moment inequalities
of Theorem~\ref{thm:BernFuncMoments}, for which the Efron-Stein
inequality \citep[Theorem 3.1]{Boucheron:2013aa} will suffice.
\begin{lemma}
\label{lem:LinMOBD}
There exists a universal
constant $C > 0$ such that, for any $f \in \mathcal{C}^3([0,1]^n)$,
\begin{multline*}
\mathbb{E}\left| f(\bv{W})-f(\bv{p})-\textstyle{\sum_{j=1}^n} \partial_j f(\bv{p})
(W_j - p_j)\right|^2 \\
\leq C \sum_{j,k=1}^n 
\|\partial_j \partial_k f\|_{\infty}^2
+ C\sum\limits_{j=1}^n \left(\sum\limits_{k=1}^n \|\partial_j\partial_k^2 f\|_{\infty}\right)^2.
\end{multline*}
\end{lemma}
\begin{proof}
Denoting $F(\bv{x}) = f(\bv{x}) - \sum_{j=1}^n \partial_j f(\bv{p}) x_j$,
from (\ref{eq:Lindeberg}), it suffices to consider
$\var F(\bv{W})$.
Let $W_j'$ be an independent copy of $W_j$ and $\bv{W}^j = (W_1,\dots,W_{j-1}, W_j', \allowbreak W_{j+1},\dots,W_n)$ for each $j=1,\dots,n$. Then there is a random vector~$\tilde{\bv{W}}^j$ such that
\[
F(\bv{W})-F(\bv{W}^j) = [\partial_j f(\bv{W})-\partial_j f(\bv{p})](W_j-W_j')
-\tfrac12 \partial_j^2 f(\tilde{\bv{W}}^j)(W_j-W_j')^2.
\]
But, from \citep[Theorem 3.1]{Boucheron:2013aa}, $\var F(\bv{W}) \leq \frac12 \sum_{j=1}^n \mathbb{E} V_j$ where
\[
V_j \, := \mathbb{E}[\cond{ \{ F(\bv{W})-F(\bv{W}^j)  \}^2 }{\bv{W}}]
\leq 2[\partial_j f(\bv{W}) - \partial_j f(\bv{p})]^2 + \tfrac12 \|\partial_j^2 f\|_{\infty}^2.
\]
The lemma now follows from Theorem~\ref{thm:BernFuncMoments} and
(\ref{eq:Lindeberg}).
\end{proof}

The final ingredient in the proof of Lemma \ref{lem:ApproxLemma2}
is a moment inequality for a sum of conditionally independent 
$\{0,1\}$-valued random variables, in which the constant does not 
depend on the number of variables. The conditional Rosenthal-type inequality
in Lemma \ref{lem:PoisBinMoments} proves effective, found by modifying the
arguments of \citep[Theorem 2.5]{Johnson:1985aa}.

\begin{lemma}
\label{lem:PoisBinMoments}
Let $X_1,\dots,X_n$ be $\{0,1\}$-valued random variables on a probability space 
$(\Omega,\mathcal{E},\mathbb{P})$ which are conditionally independent according
 to a sub-$\sigma$-algebra $\mathcal{F}$ of $\mathcal{E}$. Then, for any
 $q \geq 1$, 
\begin{equation}
\left\lVert \sum_{i=1}^n X_i \right\rVert_q \leq 2 q \left(1 + \left\lVert 
\sum_{i=1}^n \mathbb{E}[X_i\vert \mathcal{F}] \right\rVert_q\right).\label{eq:PoisBinBound}
\end{equation}
\end{lemma}

In fact, (\ref{eq:PoisBinBound}) can be improved to order $q / \log q$,
but this provides no significant improvement to our results.

\end{document}